\documentclass[leqno,11pt]{amsart}
\usepackage{amsmath,amssymb}
\usepackage{color}
\usepackage{comment}
\usepackage{todonotes}

\allowdisplaybreaks

\theoremstyle{plain}
\newtheorem{theorem}{Theorem}[section]
\newtheorem{lemma}[theorem]{Lemma}

\newtheorem{corollary}[theorem]{Corollary}
\newtheorem*{theorem-A}{Theorem A}
\newtheorem*{theorem-B}{Theorem B}
\newtheorem*{theorem-C}{Theorem C}
\newtheorem*{theorem-D}{Theorem D}
\newtheorem*{theorem-E}{Theorem E}
\newtheorem*{theorem-F}{Theorem F}
\theoremstyle{definition}

\newtheorem{remark}[theorem]{Remark}

\newtheorem{theoremalph}{Theorem}

\def\rn{\mathbb R\sp n}
\def\Rn{\mathbb R\sp n}
\def\R{\mathbb R}
\def\N{\mathbb N}

\def\M{\mathcal M}

\newcommand{\medint}{-\kern  -,435cm\int}
\newcommand{\medintinrigo}{-\kern  -,315cm\int}
\newcommand{\medelle}{-\kern  -,235cm L}
\newcommand{\medellenrigo}{-\kern  -,180cm L}

\newtoks\by
\newtoks\paper
\newtoks\book
\newtoks\jour
\newtoks\yr
\newtoks\pages
\newtoks\vol
\newtoks\publ
\newtoks\eds
\newtoks\proc
\newtoks\no
\def\ota{{\hbox{???}}}
\def\cLear{\by=\ota\paper=\ota\book=\ota\jour=\ota\yr=\ota
\pages=\ota\vol=\ota\publ=\ota}
\def\endpaper{\the\by, \textit{\the\paper},
{\the\jour} \textbf{\the\vol} (\the\yr), \the\pages.\cLear}
\def\endbook{\the\by, \textit{\the\book}, \the\publ.\cLear}
\def\endprep{\the\by, \textit{\the\paper}, \the\jour.\cLear}
\def\endproc{\the\by, \textit{\the\paper}, \the\publ, \the\pages.\cLear}
\def\name#1#2{#1 #2}
\def\et{ and }

\def\BMO{\operatorname{BMO}}
\def\VMO{\operatorname{VMO}}

\numberwithin{equation}{section}

\newcommand{\norm}[1]{{\left\vert\kern-0.25ex\left\vert\kern-0.25ex\left\vert #1
    \right\vert\kern-0.25ex\right\vert\kern-0.25ex\right\vert}}

\usepackage[normalem]{ulem}\usepackage{soul}
\usepackage{cancel}
\newcommand\Del[1]{{\color{red}\ifmmode\cancel{#1}\else\sout{#1}\fi}}

\begin{document}

\date{\today}

\title[Fractional Orlicz-Sobolev embeddings  into Campanato type spaces]{Fractional Orlicz-Sobolev embeddings \\into Campanato type spaces}

\author {Angela Alberico, Andrea Cianchi, Lubo\v s Pick and Lenka Slav\'ikov\'a}

\address{Angela Alberico, Istituto per le Applicazioni del Calcolo ``M. Picone''\\
Consiglio Nazionale delle Ricerche \\
Via Pietro Castellino 111\\
80131 Napoli\\
Italy} \email{angela.alberico@cnr.it}

\address{Andrea Cianchi, Dipartimento di Matematica e Informatica \lq\lq U. Dini"\\
Universit\`a di Firenze\\
Viale Morgagni 67/a\\
50134 Firenze\\
Italy} \email{andrea.cianchi@unifi.it}

\address{Lubo\v s Pick, Department of Mathematical Analysis\\
Faculty of Mathematics and Physics\\
Charles University\\
Sokolovsk\'a~83\\
186~75 Praha~8\\
Czech Republic} \email{pick@karlin.mff.cuni.cz}

\address{Lenka Slav\'ikov\'a,
Department of Mathematical Analysis, Faculty of Mathematics and
Physics,  Charles University, Sokolovsk\'a~83,
186~75 Praha~8, Czech Republic}
\email{slavikova@karlin.mff.cuni.cz}
\urladdr{}


\begin{abstract}
    Optimal embeddings for fractional Orlicz--Sobolev spaces into (generalized) Campanato spaces on the Euclidean space are exhibited.
    Embeddings into vanishing Campanato spaces  are also characterized. Sharp embeddings into $\BMO (\rn)$ and $\VMO (\rn)$ are derived as special instances. Dissimilarities to corresponding embeddings for classical fractional Sobolev spaces are pointed out.
\end{abstract}

\date{\today}

\maketitle

\section{Introduction}\label{intro}

The fractional Orlicz-Sobolev spaces extend the classical fractional Sobolev spaces in the sense of Gagliardo--Slobodetskij in that the power type integrability is replaced with an integrability condition governed by a general Young function.  They were introduced in \cite{BonderSalort} and their properties have been systematically investigated in various recent contributions, including \cite{ACPS_lim0,ACPS_lim1,Bal:22,CMSV,deNBS,Iok:24,Mih,Sil:21}. 
This family of spaces provides a natural functional framework for solutions to nonlocal elliptic problems associated with  non-polynomial nonlinearities -- see \cite{BaOu,BKO,BKS,CKW,FangZhang,BonderSalortVivas,OSM,SalortVivas,SalortVivas2}.

Like for any kind of Sobolev type spaces, embedding theorems are central in the theory of fractional Orlicz-Sobolev spaces, and  have been established in \cite{ACPS_emb, ACPS_inf, ACPS_modulus}. Especially, sharp embeddings into spaces defined in terms of global integrability properties of functions, called rearrangement-invariant spaces in the literature, are the subject of \cite{ACPS_emb} and \cite{ACPS_inf}. A criterion for the continuity of fractional Orlicz-Sobolev functions has been detected in \cite{ACPS_modulus}, where their optimal modulus of continuity is also exhibited.

The present paper adds one further piece of information to the theory of the spaces in question and provides a characterization of their embeddings into spaces of (generalized) Campanato type. To be more specific, denote by $V^{s,A}(\rn)$ the fractional Orlicz-Sobolev space associated with  a smoothness parameter $s\in (0,1)$ and
a Young function $A$,  and let $\mathcal L^{\varphi }(\rn)$ be the Campanato space of those functions whose integral oscillation over balls is bounded by a function $\varphi$ of their radius (precise definitions are recalled in the next section). 

Given any $s$ and $A$ as above, we exhibit the optimal (smallest) Campanato space $\mathcal L^{\varphi}(\rn)$ such that
\begin{align}\label{emb-intro1}
    V^{s,A}(\rn) \to \mathcal L^{\varphi }(\rn).
\end{align}
Here, and in what follows, the arrow \lq\lq $\to$" stands for continuous embedding and $n\in \N$.

Embeddings for higher-order spaces $V^{s,A}(\rn)$, with non-integer $s>1$, are also considered. They naturally also involve higher-order Campanato spaces $\mathcal L^{k,\varphi }(\rn)$, defined via the integral distance of functions from polynomials of order $k$. Notice that $\mathcal L^{0,\varphi }(\rn)= \mathcal L^{\varphi }(\rn)$.
Dealing with higher-order spaces calls for a normalization of trial functions. The higher-order embeddings have the form
\begin{equation}\label{emb-intro2}
 V^{s,A}_{d, k+1}(\rn) \to \mathcal{L}^{k, \varphi}(\rn),  \end{equation}
where $k \in \{0, 1, \dots, [s]-1\}$ and
$V^{s,A}_{d, k+1}(\rn)$ denotes the space of those functions in $V^{s,A}(\rn)$ all of whose derivatives of orders from  $k+1$ to $[s]$, the integer part of $s$, decay near infinity in a  weak sense. Such a normalization is required since functions which differ by a polynomial of degree between
$k+1$ and $[s]$ share the same seminorm in $V^{s,A}(\rn)$ but not in $\mathcal L^{k,\varphi }(\rn)$. 
\\ If $k=[s]$, then \eqref{emb-intro2} can be replaced with
\begin{equation}\label{emb-intro3}
 V^{s,A}(\rn) \to \mathcal{L}^{[s], \varphi}(\rn). 
 \end{equation}
Let us notice that, because of scaling reasons in $\rn$, an embedding like \eqref{emb-intro2} or \eqref{emb-intro3} can only hold if  $s<n+k-1$. The optimal 
target spaces $\mathcal{L}^{k, \varphi}(\rn)$  in \eqref{emb-intro2} and \eqref{emb-intro3} are  determined via substantially different arguments from those required for $s\in (0,1)$. Furthermore, the result has some traits which do not surface for $s\in (0,1)$. For instance,
 unlike  \eqref{emb-intro1}, the embeddings \eqref{emb-intro2} and \eqref{emb-intro3} require that $A$ satisfy a suitable decay condition near $0$, depending on $n$ and $s$. 

The use of functions $A$ and $\varphi$, which are not mere powers, discloses new features of fractional embeddings of Sobolev type into Campanato spaces built upon $\varphi$. One discrepancy from parallel embeddings into classical Campanato spaces is that the embedding \eqref{emb-intro1}, with $\varphi$ vanishing at zero, need not ensure an embedding into the space of continuous functions. Moreover, even if this is the case, the modulus of continuity of functions in $V^{s,A}(\rn)$ can be weaker than $\varphi$. This phenomenon usually emerges when $\varphi$ decays very slowly at zero --  logarithmically, for example. 

Our results about embeddings of the form \eqref{emb-intro1}
for $s\in (0,1)$ are presented in Section \ref{sec3}. The embeddings \eqref{emb-intro2}  for $s>1$ are dealt with in Section \ref{sec4} and \ref{sec5}. The most general result about \eqref{emb-intro2} and \eqref{emb-intro3} is stated in Section \ref{sec5}. In Section \ref{sec4} the basic case when $k=0$ is singled out. Although its proof follows along the same lines as in the general case,  it avoids some technicalities, and can be of use for readers who are just interested in zero-order Campanato spaces. The statements and proofs of our main results rely upon 
 several definitions and preliminary properties,  which are collected in Section \ref{S:2}.

\section{Background and preliminaries}\label{S:2}

 \subsection{Campanato type spaces}\label{2.0}

We say that a function
$\varphi\colon(0,\infty)\to(0,\infty)$, is \emph{admissible} if
\begin{equation}\label{E:admissible-new}
    \inf_{r \in [a,\infty)} \varphi(r)>0\qquad \text{for  $a\in (0,\infty)$.}
\end{equation}
Given an admissible function 
$\varphi$, the \emph{Campanato space} 
$\mathcal{L}^{\varphi}(\rn)$ is defined as the collection of all functions $u\in L^{1}_{\operatorname{loc}}(\rn)$ such that the seminorm
\begin{equation}\label{Campanato}
    |u|_{\mathcal{L}^{\varphi}(\rn)}= \sup _{B\subset \rn }
    \frac 1{\varphi (|B|^{\frac{1}{n}})} \medint_{B} |u - u_{B}| \; dx
\end{equation}
is finite, where the supremum is extended over all balls $B\subset \rn$. 
Here, $|B|$ denotes the Lebesgue measure of $B$ and $u_{B}= \medintinrigo_{B} u(x)\; dx$, where $\medintinrigo_{B}$ stands for the integral mean over $B$.

We also consider the \emph{vanishing Campanato space} $V\!\mathcal{L}^\varphi(\rn)$ defined as the subset of $\mathcal{L}^{\varphi}(\rn)$  of those functions $u\in \mathcal{L}^{\varphi}(\rn)$   such that, in addition,
\begin{equation}\label{def_V}
    \lim_{r\to 0^+}  \sup_{B\subset \rn, |B|\leq r} \; \frac 1{\varphi (|B|^{\frac{1}{n}})}\medint_{B} |u - u_{B}| \; dx =0\,.
\end{equation}
In the particular case when $\varphi(t)=1$, the space $\mathcal{L}^{\varphi}(\rn)$ is the set of all functions of \emph{bounded mean oscillation}, denoted by $\BMO (\rn)$, and the corresponding space $V\!\mathcal{L}^{\varphi}(\rn)$ is the collection of all functions of \emph{vanishing mean oscillation} denoted by $\VMO (\rn)$. 

Given $k\in\N$, we define the \emph{$k$-th order Campanato space} 
$\mathcal{L}^{k,\varphi}(\rn)$  as the collection of all functions $u\in   W^{k,1}_{\operatorname{loc}}(\rn)$ such that the seminorm
\begin{equation}\label{Camp_k}
|u|_{\mathcal{L}^{k,\varphi}(\rn)}
    = 
    \sup_{B\subset\rn} 
    \frac{1}{\varphi (|B|^{\frac{1}{n}})\, 
    |B|^{\frac{k}{n}}}
    \medint_{B}|u-P^k_{B}[u] | \; dx
\end{equation}
 is finite, where $P^k_{B} [u]$ is the (unique)  polynomial of degree at most $k$ such that
\begin{equation}\label{Poli_cond}
  \int_{B} \left( \nabla ^h u - \nabla ^h P^k_{B} [u] \right)\; dx =0 \qquad \hbox{for $h=0, \ldots , k$}.
\end{equation}
The $k$-th order vanishing Campanato space $V\!\mathcal{L}^{k,\varphi}(\rn)$ is defined accordingly.

 \subsection{Young functions}\label{2.1}
  Let   $A\colon  [0,
\infty ) \to [0, \infty ]$ be a \emph{Young function}, namely a convex,  left-continuous function, which is non-constant in $(0,\infty)$ and such that $A(0)=0$.
Note that the function
\begin{align}
    \label{monotone}
    \frac{A(t)}{t} \quad \text{is non-decreasing in $(0, \infty)$.}
\end{align}
Moreover,
\begin{equation}\label{Ak}
\lambda A(t) \leq A(\lambda t) \qquad \text{for $\lambda >1$ and $t\geq 0$.}
\end{equation}
The \emph{Young conjugate} $\widetilde{A}$   of $A$   is also a Young function and is defined as
\begin{equation}\label{2.8}
\widetilde{A}(t) = \sup \{\tau t-A(\tau ):\,\tau \geq 0\}  \qquad \text{for $t\geq 0$.}
\end{equation}
One has that
\begin{equation}\label{AAtilde}
t \leq A^{-1}(t) \widetilde A^{-1}(t) \leq 2t \qquad \text{for $t\geq 0$,}
\end{equation}
where generalized inverses are defined as to be right-continuous. 
\\
As a~consequence of~\eqref{Ak}, one also has that
\begin{equation}\label{min}
    \min\{1, \lambda \}A^{-1}(t) \leq A^{-1}(\lambda t) \leq \max\{1, \lambda \}A^{-1}(t) \qquad\text{for $\lambda , t \geq 0$.}
\end{equation}
 A  Young function $A$ is said to \emph{dominate} another Young function $B$ \emph{globally} [resp. \emph{near zero}] [resp. \emph{near infinity}]
   if there exist positive
 constants $c$ and $t_0$ such that
\begin{equation}\label{B.5bis}
B(t)\leq A(c t) \qquad \text{for $ t\geq 0$ \, [for $0\leq  t\leq t_0$] \, [for $t\geq t_0$]}.
\end{equation}
The functions $A$ and $B$ are called \emph{equivalent  globally} [\emph{near zero}] [\emph{near infinity}],   if they dominate each other in the respective range of values of their argument.
We shall write
$$A \simeq B$$ to denote that $A$ is equivalent to $B$.
\\
By contrast,  
 we shall use the relation $B\lesssim A$ between two (not necessarily Young) functions to denote that $B$ is bounded by $A$ up to positive multiplicative constants depending on appropriate quantities; the 
 relation $ A \approx B$ means that both $B\lesssim A$ and $A\lesssim B$.
\\
The \emph{Matuszewska-Orlicz indices} 
of a Young function $A$ near $0$ and near infinity are defined as
\begin{equation}\label{index}
   i_0(A) 	= \lim_{\lambda \to 0^+} \frac{\log \lambda}{\log \Big(\liminf_{t\to0^+} \frac{A^{-1}(\lambda t)}{A^{-1}(t)} \Big)} \qquad \text{and} \qquad    i_\infty(A) 	= \lim_{\lambda \to 0^+} \frac{\log \lambda}{\log \Big(\liminf_{t\to\infty} \frac{A^{-1}(\lambda t)}{A^{-1}(t)} \Big)}.
\end{equation}
If $A$ is finite-valued and vanishes only at $0$, then the following alternative expressions for $i_0(A)$ and $i_\infty(A)$ hold:
\begin{equation}\label{indexbis}
i_0(A) = \lim_{\lambda \to \infty} \frac{\log \Big(\liminf _{t\to 0^+}\frac{A(\lambda t)}{A(t)}\Big)}{\log \lambda} \qquad \text{and} \qquad   i_\infty(A) =\lim_{\lambda \to \infty} \frac{\log \Big(\liminf _{t\to \infty}\frac{A(\lambda t)}{A(t)}\Big)}{\log \lambda}.
\end{equation}
 In what follows, we shall make use of the fact that, if $A$ is a Young function and $s\in (0,n)$,  then
\begin{align}
    \label{equiv-cond}
     \int_0 \left(\frac t{A(t)}\right)^{\frac {s}{n-s}} \; dt<\infty \qquad \text{if and only if} \qquad \int_0 \frac{\widetilde A (t)}{t^{1+\frac{n}{n-s}}}\, dt<\infty,
\end{align}
and 
\begin{align}
    \label{equiv-cond-inf}
     \int^\infty \left(\frac t{A(t)}\right)^{\frac {s}{n-s}} \; dt<\infty \qquad \text{if and only if} \qquad \int^\infty\frac{\widetilde A (t)}{t^{1+\frac{n}{n-s}}}\, dt<\infty.
\end{align}

\subsection{Orlicz spaces and rearrangement-invariant spaces}
Let $\Omega$ be a measurable set in $\rn$. We denote by $\mathcal M(\Omega)$ the set of all  measurable real-valued functions on $\Omega$. The \emph{Orlicz space} $L^A(\Omega)$ is the Banach space of all    functions $u\in \mathcal M(\Omega)$ for which the \emph{Luxemburg norm} defined by
\begin{equation}\label{lux}
     \|u\|_{L^A(\Omega)}= 
     \inf \left\{ \lambda >0 :  \int_{\Omega}A
    \left( \frac{|u|}{\lambda} \right) \,dx
    \leq 1 \right\}\,
\end{equation}
is finite. In particular, $L^A (\Omega)= L^p (\Omega)$ if $A(t)=
t^p$ for some $p \in [1, \infty )$, and $L^A (\Omega)= L^\infty
(\Omega)$ if $A(t)=0$ for $t\in [0, 1]$ and $A(t) = \infty$ for
$t\in (1, \infty)$.
\\ 
If $E\subset\Omega$ is  a measurable set, then
\begin{equation}\label{E:norm-characteristic}
    \|\chi_E\|_{L^A(\Omega)}=
    A^{-1}\left(\frac{1}{|E|}\right).
\end{equation}
\\
One has that
\begin{equation}\label{dualA}
\|v\|_{L^{\widetilde A}(\Omega)} \leq \sup_{u \in L^A(\Omega)}\frac{\displaystyle \int_{\Omega} u v \, dx}{\|u\|_{L^A(\Omega)}} \leq 2\, \|v\|_{L^{\widetilde A}(\Omega)}
\end{equation}
for $v \in L^{\widetilde A}(\Omega)$. In the special case when $\Omega = (\alpha,\beta) \subset \R$ for some $\alpha < \beta$,   and $f\colon (\alpha,\beta) \to [0, \infty)$ is a non-increasing function, one has that  
\begin{equation}\label{E:dual-monotone}
    \|f\|_{L^{\widetilde A}(\alpha,\beta)} 
    \leq \sup_{g \in L^A(\alpha,\beta),\ g\downarrow}\frac{\displaystyle \int_{\alpha}^{\beta} f(t)\, g(t)\, dt}{\|g\|_{L^A(\alpha,\beta)}} \leq 2\, \|f\|_{L^{\widetilde A}(\alpha,\beta)}.
\end{equation}
This is a consequence of  the Hardy-Littlewood inequality for rearrangements.
Here, the notation $g\downarrow$ indicates that  the supremum is restricted to nonnegative functions which are non-increasing on $(\alpha,\beta)$. 

 The Orlicz spaces belong to the family of rearrangement-invariant spaces, whose definition relies upon the notion of decreasing rearrangement of a function.
The
 \emph{decreasing
rearrangement} $u^{\ast}\colon [0, \infty) \to [0, \infty]$ of a function $u\in \mathcal M (\Omega)$ is given by
\begin{equation}\label{u*}
u^{\ast}(r) = \inf \{t\geq 0: |\{x\in \Omega: |u(x)|>t \}|\leq r \} \qquad \text{for $r\geq 0$.}
\end{equation}
In other words, $u^{\ast}$
is the (unique) non-increasing,
right-continuous function
which is equidistributed with $u$. Given  $u\in \mathcal M (\Omega)$, we also define the function $u^{**}\colon (0, \infty) \to [0, \infty]$ by
\begin{equation*}
    u^{**}(r)
    =
    \frac1r \int_{0}^r
    u^*(\varrho)\,d\varrho
    \qquad
    \text{for $r>0$.}
\end{equation*}

A \emph{rearrangement-invariant space} is a
Banach function space $X(\Omega)$, in the sense of Luxemburg \cite[Chapter 1, Section 1]{BS},   such that
\begin{equation}\label{B.1}
 \|u\|_{X(\Omega)} = \|v \|_{X(\Omega)} \qquad \text{whenever $u^*=v^*$.}
 \end{equation}
The \emph{associate space} $X^{'}(\Omega)$ of $X(\Omega)$ is the rearrangement-invariant
space of all functions in $\mathcal M(\Omega)$ for which the norm
\begin{equation}\label{B.2}
 \|u \|_{X^{'}(\Omega)} =
\sup_{v \neq 0} \frac{\displaystyle\int _{\Omega}|u v| \;dx}{\|v \|_{X(\Omega)}}
\end{equation}
is finite.
As a consequence, the H\"older type inequality
\begin{equation}\label{holder}
\int _{\Omega} |u(x) v(x)|\,dx \leq \|u\|_{X(\Omega)}
\|v\|_{X^{'}(\Omega )}
\end{equation}
holds for every $u \in X(\Omega)$ and $v\in X^{'}(\Omega)$. Note that, owing to \eqref{dualA}, if $A$ is a Young function and $X(\Omega)=L^A(\Omega)$, then $X'(\Omega)=L^{\widetilde A}(\Omega)$, up to equivalent norms.
\\
Given a  rearrangement-invariant space $X(\Omega)$,  there exists a  rearrangement-invariant space on $(0,\infty)$, denoted by $X(0,\infty)$ and called the \emph{representation space} of $X$, such that
\begin{equation}\label{B.3}
\|u \|_{X(\Omega)} = \|u^{\ast} \|_{X(0,\infty)}
\end{equation}
for every $u\in X(\Omega)$.

\subsection{Integer-order Sobolev type spaces}
 Given $m \in \N$ and a  Young function $A$, we
denote by $V^{m,A}( \rn)$ the  \emph{homogeneous Orlicz-Sobolev space} given by
\begin{equation}\label{homorliczsobolev}
V^{m,A}(\rn ) = \{ u\in \mathcal M(\rn):  \text{$u$ is   $m$-times weakly differentiable in $ \rn$ and   $|\nabla ^m u| \in L^A( \rn)$}\}.
\end{equation}
The alternate notation $V^mL^A(\rn)$ for $V^{m,A}(\rn )$ will also be employed when convenient.
The functional
$$\|\nabla ^m u\|_{L^A(\rn)}$$
defines a seminorm on the space $V^{m,A}(\rn)$.
\\
More generally, the \emph{$m$-th order homogeneous Sobolev space} $V^mX(\rn)$ associated with 
 a rearrangement-invariant space 
$X(\rn)$ is defined as 
\begin{equation}\label{homrisobolev}
V^{m}X(\rn ) = \{ u\in \mathcal M(\rn):  \text{$u$ is   $m$-times weakly differentiable in $ \rn$ and   $|\nabla ^m u| \in X( \rn)$}\}.
\end{equation}
The space $V^mX(\rn)$ is equipped with the seminorm $\|\nabla ^m u\|_{X(\rn)}$.

The following result provides us with an optimal Sobolev embedding into Campanato type spaces and is a consequence of  \cite[Theorem 1.1]{CP-2003}.
  Although that theorem is stated for function spaces defined on cubes, its proof applies also to the present version in $\rn$.

 \begin{theoremalph}\label{CP} \emph {Let $X(\rn)$ be a rearrangement-invariant space and let $\varphi_X \colon  (0,\infty) \to [0, \infty]$ be the function defined by
\begin{equation}\label{varpi}
 \varphi_X (r) = r^{-n}\,\|\varrho^{\frac 1n} \chi_{(0, r^n)} (\varrho) \|_{X' (0, \infty)}\qquad \text{for $r > 0$.}
 \end{equation}
 Then,
\begin{equation}\label{CP1}
     V^1X(\rn) 
     \to 
     \mathcal L^{\varphi_X}(\rn)
 \end{equation}
 and
\begin{equation}\label{CP1-bis}
     |u|_{\mathcal L^{\varphi_X}(\rn)}
     \leq c\,
     \|\nabla u\|_{X(\rn)}
 \end{equation}
for some constant   $c=c(n)$ and every $u  \in V^1X(\rn)$. Moreover, $\mathcal L^{\varphi_X}(\rn)$ is the optimal Campanato target space in \eqref{CP1} and~\eqref{CP1-bis}.}
 \end{theoremalph}

 \subsection{Fractional Orlicz-Sobolev spaces}\label{2.2}

Given $s\in (0,1)$,
the seminorm $|u|_{s,A, \rn}$ is  defined as
\begin{equation}\label{aug340}
|u|_{s,A, \rn}=  \inf\Big\{\lambda>0: J_{s,A}\Big(\frac u{\lambda}\Big)\leq 1\Big\}
\end{equation}
for $u \in \mathcal M (\rn)$, where $J_{s,A}$ is the functional given by
\begin{equation}\label{intro1}
    J_{s,A}(u)=
\int_{\rn} \int_{\rn}A\left(\frac{|u(x)-u(y)|}{|x-y|^s}\right)\frac{\,dx\,dy}{|x-y|^n}.
\end{equation}
The \emph{homogeneous
 fractional Orlicz-Sobolev space} $V^{s,A}(\rn)$ is defined as
\begin{equation}\label{aug341}
	V^{s,A}(\rn) = \big\{u \in \mathcal M (\rn):  |u|_{s,A, \rn}<\infty\}\,.
\end{equation}
The definitions of the seminorm $|u|_{s,A, \rn}$ and of the space $V^{s,A}(\rn)$ carry over to vector-valued functions $u$ just by replacing the absolute value of $u(x)-u(y)$ by the Euclidean norm  of the same expression in the definition of the functional $J_{s,A}$.
\\
The definition of $V^{s,A}(\rn)$ is extended to all $s\in (0, \infty) \setminus \N$ in a standard way. Namely, on
denoting, as above, by $[s]$ and  $\{s\}$ the integer part and the fractional part of $s$, respectively, so that
$$\{s\}= s-[s],$$
we define
\begin{equation}\label{aug343}
V^{s,A}(\rn ) = \{u \in \mathcal M(\rn): \text{$u$ is   $[s]$-times weakly differentiable in $\rn$ and  $ \nabla ^{[s]}u \in V^{\{s\}, A}(\rn)$}\}\,.
\end{equation}
Given $k \in \{0, 1, \dots , [s]\}$, 
the subspace $V^{s,A}_{d,k}(\rn)$ of those functions in $V^{s,A}(\rn)$ whose derivatives of orders between $k$ and $[s]$ decay near infinity 
is defined as
\begin{equation}\label{nov100higher1}
	V^{s,A}_{d,k}(\rn) =  \{ u \in V^{s,A}(\rn): |\{|\nabla ^h u|>t\}|<\infty\,\ \text{for  every $t>0$ and for $h=k, \dots ,[s]$}\}.
\end{equation}
Here, $\nabla ^0u$ stands for $u$.

In dealing with embeddings of fractional Orlicz-Sobolev spaces with smoothness parameter $s>1$ into Campanato spaces, a critical role is played by embeddings of the same kind of spaces into the corresponding optimal rearrangement-invariant target spaces. The latter embeddings were exhibited in \cite{ACPS_emb}. For the present use, only the  associate space of the optimal rearrangement-invariant target space is needed. A formula for its norm is provided by the next result.

\begin{theorem}\label{thm_Z}
    Let  $s\in (0, n) \setminus \N$ and let $A$ be a Young function such that
\begin{equation}\label{intzero'}
    \int_0 
    \left(\frac t{A(t)}\right)
    ^{\frac {s}{n-s}} \; dt<\infty.
\end{equation}
    Define the functional $\|\cdot\|_{Z_s(0,\infty)}$ as
    \begin{equation}\label{emb-ass}
        \|f\|_{Z_s(0,\infty)}
        = \|r^{\frac{s}{n}}f^{**}(r)\|_{L^{\widetilde{A}}(0,\infty)}
    \end{equation}
    for $f\in\M(0,\infty)$. Then, $\|\cdot\|_{Z_s(0,\infty)}$ is a~rearrangement-invariant norm on $(0,\infty)$. Denote by $\|\cdot\|_{X_s(0,\infty)}$ the norm defined by
    \begin{equation*}
        \|f\|_{X_s(0,\infty)}
        = \|f\|_{Z_s'(0,\infty)}
    \end{equation*}
    for $f\in\M(0,\infty)$. Then,
    \begin{equation}\label{E:emb-ri}
        V_{d,0}^{s,A}(\rn)
        \to
        X_s(\rn)
    \end{equation}
    and 
\begin{equation}\label{emb}
        \| u \|_{ X_s(\rn)} \leq c\, | \nabla^{[s]} u|_{\{s\}, A,\rn}
    \end{equation}
    for some constant $c=c(n,s)$ and every $u\in V_{d,0}^{s,A}(\rn)$. 
    Moreover, $X_s(\rn)$ is the optimal rearrangement-invariant target space in \eqref{E:emb-ri} and~\eqref{emb}.
\end{theorem}

\begin{proof} 
     By equation \eqref{equiv-cond}, the  condition~\eqref{intzero'} is  equivalent to
    \begin{align}\label{fi2}
         \int_0 \frac{\widetilde A (t)}{t^{1+\frac{n}{n-s}}}\, dt<\infty.
    \end{align}
    The change of variables
    \begin{equation*}
        t = (1+r)^{-1+\frac{s}{n}}
    \end{equation*}
    shows that~\eqref{fi2} holds if and only if 
    \begin{align}\label{fi1}
        \int_{0}^{\infty}\widetilde{A}\left((1+r)^{-1+\frac{s}{n}}\right)\,dr<\infty.
    \end{align}
    Hence, by the definition of the Luxemburg norm, we have
        \begin{equation*}
        \|(1+r)^{-1+\frac{s}{n}}\|_{L^{\widetilde{A}}(0,\infty)}<\infty.
    \end{equation*}
    Owing to ~\cite[Theorem~E]{ACPS_emb},  the functional $\|\cdot\|_{Z_s(0,\infty)}$ is a~rearrangement-invariant norm, there exists a constant $c=c(n,s)$  such that
    \begin{equation}\label{july100}
        \left\|
        \int_{r}^{\infty}
        \varrho^{-1+\frac{s}{n}}f(\varrho)\,d\varrho
        \right\|_{X_s(0,\infty)}
        \le c\, \|f\|_{L^{A}(0,\infty)} 
    \end{equation}
    for $f\in L^A(0,\infty)$, and $X_s(0,\infty)$ is the optimal  rearrangement-invariant target space in \eqref{july100}.
    Therefore, 
    \cite[Theorem 3.7]{ACPS_inf} implies that the inequality~\eqref{emb} holds for $u\in V_{d,0}^{s,A}(\rn)$, and that $X_s(\rn)$ is the optimal rearrangement-invariant target space in \eqref{emb}.
\end{proof}

\section{Case $s\in (0,1)$}\label{sec3}

We begin by focusing  on embeddings of the form \eqref{emb-intro1}
 for $s\in (0,1)$. Given  a Young function $A$, the optimal Campanato target space in \eqref{emb-intro1} is associated with the function $\varphi_{s, A} \colon  [0, \infty) \to [0, \infty)$ defined as
\begin{equation}\label{1}
\varphi_{s, A} (r) = r^s A^{-1}  (r^{-n})\qquad \text{for } r> 0\,.
\end{equation}
This is the content of Theorem \ref{teo1} below, where uniform embeddings into vanishing Campanato spaces are also characterized. 
We say that $$V^{s, A}(\rn) \to V\!\mathcal{L}^{\varphi}(\rn) \qquad \text{uniformly}$$  if   $V^{s, A}(\rn) \to \mathcal L^{\varphi}(\rn)$ and
\begin{equation}\label{E:vanishing}
    \lim_{r\to 0^+}
    \sup_{|u|_{s, A, \rn} \leq 1 } 
    \left( 
    \sup_{B\subset \rn, |B| \leq r} \; \frac 1{\varphi (|B|^{\frac{1}{n}})}\medint _B |u - u_B| \; dx
    \right)=0.
\end{equation}

\begin{theorem}[\bf Optimal embeddings into Campanato spaces for $s\in (0,1)$]\label{teo1}
Assume that $s\in (0,1)$. Let  $A$ be a Young function and let $\varphi_{s, A}$ be the function defined as in \eqref{1}.
\\ (i) The embedding
\begin{equation}\label{2}
V^{s, A}(\rn) \to \mathcal{L}^{\varphi_{s,A}}(\rn)
\end{equation}
holds, and there exists a constant $c=c(n,s)$ such that
\begin{equation}\label{3}
|u|_{\mathcal{L}^{\varphi_{s,A}}(\rn)} \leq c \,|u|_{s, A, \rn}
\end{equation}
for every $u\in V^{s, A}(\rn)$.
Moreover, $\mathcal{L}^{\varphi_{s,A}}(\rn)$ is the optimal  Campanato target space in \eqref{2} and \eqref{3}.
\\ 
(ii) Let $\varphi$ be an admissible function. Then,
\begin{equation}\label{2v}
    V^{s, A}(\rn) 
    \to 
    V\!\mathcal{L}^{\varphi}(\rn)\qquad \text{uniformly}
\end{equation}
if and only if
 \begin{equation}\label{limit1}
\varphi_{s,A}(r)\leq c \, \varphi (r) \qquad \text{for $r>0$,}
\end{equation}
for some positive constant $c$
and
\begin{equation}\label{limit}
\lim_{r \to 0^+} \frac{\varphi_{s,A}(r)}{\varphi (r)} =0.
\end{equation}
\end{theorem}

Let us take the chance to comment on a gap between embeddings into Campanato spaces and spaces of uniformly continuous functions for general  spaces $V^{s,A}(\rn)$.
Theorem \ref{teo1}, Part (i),  can be used to prove an embedding of a fractional Orlicz--Sobolev space $V^{s,A}(\rn)$ into 
$\mathcal C^{\varphi (\cdot)} (\rn)$ for some modulus of continuity $\varphi$,  under suitable assumptions on $s$ and $A$, via a result  from  \cite{Spanne}. However, as already mentioned in Section \ref{intro}, the conclusions that follow from this approach are not sharp in general. This gap between Campanato and H\"older type spaces does not surface in embeddings of classical Sobolev spaces.
\\ Specifically,  \cite[Corollaries 1 and 2]{Spanne} imply that
\begin{equation}\label{19}
\mathcal{L}^{\varphi_{s, A}} (\rn) \to \mathcal C^{\varphi (\cdot)} (\rn)
\end{equation}
 for some $\varphi$
if and only if
\begin{equation}\label{20}
    \int_0 \frac{\varphi_{s,A}(r)}{r}\; dr < \infty\,
\end{equation}
and the embedding \eqref{19} holds with 
\begin{align}
    \label{july102}
    \varphi(r)= \int_0^r \frac{\varphi_{s,A}(\rho)}{\rho}\; d\rho.
\end{align}
 The condition \eqref{20} reads
$$
\int_0 A^{-1}\left( \frac 1{t^n}\right )\, t^{s-1} \; dt<\infty\,,
$$
which, by a change of variables, turns into
\begin{equation}\label{21}
\int^\infty \frac {A^{-1} (t)}{t^{1+\frac sn}}
\; dt <\infty\,.
\end{equation}
However, the condition~\eqref{21} is in general essentially stronger than
\begin{equation}\label{E:weaker}
    \int^{\infty} 
    \left (\frac t{A(t)}\right)^{\frac s{n-s}} \; dt <\infty \,,
\end{equation}
which is, owing to~\cite[Theorem~4.2]{ACPS_modulus},
the optimal condition for an  embedding of the form
\begin{equation}\label{22}
    V^{s,A}(\rn) \to  \mathcal C^{\varphi (\cdot)} (\rn)
\end{equation}
 to hold for some $\varphi$. \\ The gap between \eqref{21} and \eqref{22} can be demonstrated, for instance, by considering a Young function $A$ satisfying
\begin{equation}\label{2024-50}
    A(t)\approx t^{\frac{n}{s}}
    \left(\log t\right)^{\alpha}
    \qquad\text{for $t$ near infinity,}
\end{equation}
where $\alpha\in\R$. For this choice of $A$,  the condition \eqref{21} is equivalent to $\alpha>\frac{n}{s}$, whereas~\eqref{E:weaker} amounts to $\alpha>\frac{n}{s}-1$. Furthermore,  in the common range of exponents  $\alpha>\frac{n}{s}$, formula \eqref{july102}
implies \eqref{22} with $\varphi(r) \approx \left(\log\left( 1+ \frac 1r\right)\right)^{1-\frac{s\alpha}{n}}$, whereas 
\cite[Example 4.8]{ACPS_modulus} yields the stronger and sharp conclusion that \eqref{22} holds for 
$\varphi (r)\approx \left(\log\left( 1+ \frac 1r\right)\right)^{1-\frac{s(\alpha+1)}{n}}$.
\\
In particular, this shows that, unlike the case of standard fractional Sobolev spaces, iteration does not always yield optimal results if embeddings  of fractional  Orlicz-Sobolev spaces into 
spaces of continuous functions
 are factored via (yet optimal) embeddings into Campanato type spaces.
 An analogous assertion concerns embeddings into $L^\infty(\rn)$.

\medskip
The following characterization of fractional Orlicz-Sobolev embeddings into $\BMO(\rn)$ and $\VMO(\rn)$ can be deduced from  Theorem \ref{teo1}. 

\begin{corollary}[\bf Embeddings into $\BMO$]\label{corollary1}
Let $s\in (0,1)$ and let $A$ be a Young function.
\\ (i)  One has that
\begin{equation}\label{BMO_emb}
V^{s,A}(\rn) \to \BMO(\rn)
\end{equation}
 if and only if
\begin{equation}\label{BMO_cond}
 \frac{A(t)}{t^{\frac ns}} \geq c \quad \quad\text{for $t>0$,}
\end{equation}
for some positive constant $c$.
\\ (ii)  One has that
\begin{equation}\label{VMO_emb}
V^{s,A}(\rn) \to VMO(\rn) \quad \text{uniformly}
\end{equation}
if and only if   \eqref{BMO_cond}
holds and \begin{equation}\label{VMO_cond}
\lim_{t \to \infty} \frac{A(t)}{t^{\frac ns}} = \infty.
\end{equation}
\end{corollary}

\begin{proof}[Proof of Theorem \ref{teo1}] Throughout this proof, the relations $\lesssim$ and $\approx$ hold up to constants depending on $n$ and $s$. 
\\
Part (i). 
Let $u\in V^{s,A}(\rn)$, $t\in\R$, $r>0$ and let $B_r$ denote any ball in $\rn$ with radius $r$. By Jensen's inequality,
\begin{align}\label{4}
A\left (\frac{|t - u_{B_r}|}{(2r)^s} \right)  = A\left( \frac 1{|B_r|} \left|
\int_{B_r}\frac{t - u(y)}{(2r)^s}\; dy \right| \right)
\leq \frac 1{|B_r|} \int_{B_r} A\left(\frac{|t - u(y)|}{(2r)^s}\right)\; dy\,.
\end{align}
Applying  the inequality \eqref{4} with $t =u(x)$ and integrating with respect to $x$ over $B_r$ yields
\begin{align}\label{5}
&\int_{B_r} A\left (\frac{|u(x) - u_{B_r}|}{(2r)^s} \right)\; dx
\\ \cr & \leq \frac 1{|B_r|} \int_{B_r}  \int_{B_r} A\left (\frac{|u(x) - u(y)|}{(2r)^s} \right)\; dx \,dy
 \leq \frac {(2r)^n}{|B_r|} \int_{B_r}  \int_{B_r} A\left (\frac{|u(x) - u(y)|}{|x-y|^s} \right)\; \frac{dx\, dy}{|x-y|^n} \nonumber
 \\ \cr & \lesssim  \int_{B_r}  \int_{B_r} A\left (\frac{|u(x) - u(y)|}{|x-y|^s} \right)\; \frac{dx \, dy}{|x-y|^n}
 \lesssim \int_{\rn}  \int_{\rn} A\left (\frac{|u(x) - u(y)|}{|x-y|^s} \right)\; \frac{dx\, dy}{|x-y|^n}. \nonumber
\end{align}
Fix $\lambda>|u|_{s,A,\rn}$. An application of \eqref{5} with  $u$ replaced with $u/\lambda$ tells us that
\begin{equation}
    \int_{B_r} A\left (\frac{|u(x) - u_{B_r}|}{\lambda(2r)^s} \right)\; dx
        \lesssim
    J_{s,A}\Big(\frac u{\lambda}\Big).
\end{equation}
Therefore, owing to \eqref{aug340}, \eqref{lux} and~\eqref{Ak},
\begin{align}\label{7}
 \| u -u_{B_r}\|_{L^A(B_r)} &\lesssim\, r^s\,
 |u|_{s,A,\rn}\,.
\end{align}
On the other hand, by~\eqref{holder} and~\eqref{min}, 
\begin{align}
\int_{B_r} | u - u_{B_r}| \; dx & \leq   2 \, \| u - u_{B_r}\|_{L^A(B_r)} \, \|1\|_{L^{\widetilde{A}}(B_r)}\label{8}
\\  &
= 2 \, \| u -u_{B_r}\|_{L^A(B_r)} \, \frac 1{\widetilde{A}^{-1} \left ( \frac 1{|B_r|}\right)} \approx  \, \| u - u_{B_r}\|_{L^A(B_r)}  \, \frac 1{\widetilde{A}^{-1} \left ( \frac 1{r^n}\right)} \,.\nonumber
\end{align}
Combining  the inequalities \eqref{7} and \eqref{8} implies that
\begin{equation}\label{11}
\frac{\widetilde{A}^{-1} \left(\frac 1{r^n}\right)}{r^s} \int_{B_r} |u - u_{B_r}|\; dx \lesssim\, |u|_{s,A, \rn}.
\end{equation}
Thanks to~\eqref{AAtilde} and~\eqref{1}, the latter inequality can be rewritten as
\begin{equation}\label{11bis}
\frac{1}{\varphi_{s,A}(r)}\medint_{B_r} |u - u_{B_r}|\; dx \lesssim\, |u|_{s,A, \rn}\,.
\end{equation}
 The inequality \eqref{3} hence follows by
taking the supremum over all balls $B_r \subset \rn$.
\\
In order to prove the optimality of the function $\varphi_{s,A}$ in~\eqref{2} and~\eqref{3}, suppose that $\varphi$ is an admissible function such that
\begin{equation}\label{3varphi}
    |u|_{\mathcal{L}^{\varphi}(\rn)} 
    \leq c \, |u|_{s, A, \rn}
\end{equation}
for $u\in V^{s,A}(\rn)$.  Given any non-increasing function $f\colon  [0, \infty)\to [0, \infty)$,  define the function $u_f: \rn \to \R$ as
\begin{equation}\label{uf}
u_f(x)= \int_{\omega_n|x|^n}^{\infty} f(r)\, r^{-1+ \frac sn} \; dr \qquad \text{for } x\in \rn\,.
\end{equation}
Hence,
\begin{equation}\label{13}
u_f^*(r) = \int_r^\infty f(\varrho)\, \varrho^{-1 +\frac sn}\; d \varrho \qquad \hbox{for}\;\; r> 0\,,
\end{equation}
and, by \cite[Inequality (6.15)]{ACPS_emb},
\begin{equation}\label{12}
|u_f|_{s,A,\rn} \lesssim  \, \|f\|_{L^A(0, \infty)}\,.
\end{equation}
Fix any ball $B\subset\rn$ centered at $0$. Assume that ${\rm supp} \,f \subset (0, |B|)$, 
whence ${\rm supp}\,
 u_f \subset B$.
Then,  
\begin{align}\label{15}
\int_{B} |u_f - (u_f)_B| \; dx &\geq 
\int_B |u_f(x) - {\textrm{med}}_B(u_f) |\; dx 
\geq\int_0^{\tfrac{|B|}2} \left(u_f^* (r) - u_f^*\big(\tfrac{|B|}2\big)\right ) \; dr 
 \\
 & = \int_0^{\tfrac{|B|}2}  r\big( - {u_f^*}^{'}(r)\big) \; dr= \int_0^{\tfrac{|B|}2} f(r)\, r^{\frac sn}\, dr\,,
\nonumber
\end{align}
where ${\textrm{med}}_B(u_f) = u_f^*\big(\tfrac{|B|}2\big)$, the median of $u_f$ over $B$.
Note that the first inequality in \eqref{15} holds by the minimizing property of the median for the distance in $L^1$ from the space of constant functions, the first equality by the equimeasurability of $u_f$ and $u_f^*$, and the 
 second equality via integration by parts. Now, choose 
$f=\chi_{(0,|B|)}$.
Thus, by~\eqref{3varphi}, \eqref{Campanato}, \eqref{12}, \eqref{15}, \eqref{E:norm-characteristic} and~\eqref{1}, 
\begin{align*}
    c &\ge \frac{|u_f|_{\mathcal L^{\varphi}(\rn)}}{|u_f|_{s,A,\rn}}
    \ge 
    \frac{1}{|u_f|_{s,A,\rn}}
    \frac{1}{\varphi(|B|^{\frac{1}{n}})|B|}
    \int_B|u_f-(u_f)_B|\,dx
        \\
    &\gtrsim
    \frac{1}{\|f\|_{L^A(0,\infty)}}
    \frac{1}{\varphi(|B|^{\frac{1}{n}})|B|}
    \int_0^{\tfrac{|B|}2} f(r)
    \, r^{\frac sn}\, dr
        \\
    &\approx
    \frac{A^{-1} (1/|B|)|B|^{
    \frac{s}{n}}}{\varphi(|B|^{\frac{1}{n}})}
    =
    \frac{\varphi_{s,A}(|B|^{\frac{1}{n}})}{\varphi(|B|^{\frac{1}{n}})}.
\end{align*}
Thanks to the arbitrariness of $B$, this yields
\begin{equation*}
    \varphi_{s,A}(r)
    \le c\,
    \varphi(r)
    \qquad\text{for $r>0$,}
\end{equation*}
whence the optimality of $\varphi_{s,A}$ in~\eqref{2} and~\eqref{3} follows.
\\ Part (ii).  Assume that $\varphi$ is an admissible function such that  \eqref{limit1} and \eqref{limit} hold. Thanks to Part (i), the former equation ensures that
the embedding \eqref{2} holds. Moreover, the latter equation implies that
\begin{align*}
    \lim_{r\to0^+}
    &\sup_{|u|_{s,A,\rn}\le1}
    \sup_{B\subset\rn,\ |B|\le r}
    \frac{1}{\varphi(|B|^{\frac{1}{n}})}
    \medint_B
    |u-u_B|\,dx
        \\
    &=
    \lim_{r\to0^+}
    \sup_{B\subset\rn,\ |B|\le r}
    \frac{\varphi_{s,A}(|B|^{\frac{1}{n}})}{\varphi(|B|^{\frac{1}{n}})}
    \sup_{|u|_{s,A,\rn}\le1}
    \frac{1}{\varphi_{s,A}(|B|^{\frac{1}{n}})}
    \medint_B
    |u-u_B|\,dx
        \\
    &\lesssim
    \lim_{r\to0^+}
    \sup_{B\subset\rn,\ |B|\le r}
    \frac{\varphi_{s,A}(|B|^{\frac{1}{n}})}
    {\varphi(|B|^{\frac{1}{n}})}
    =0.
\end{align*}
Hence, \eqref{2v} follows. 
\\ Conversely, assume that $\varphi$ is an admissible function such that the embedding \eqref{2v} holds. By Part (i), the condition \eqref{limit1} must be fulfilled. As for the limit \eqref{limit},
note that
\begin{align}
\frac{\varphi_{s,A}(r)}{\varphi(r)}
    &=
    \frac{r^sA^{-1}(r^{-n})}{\varphi(r)} 
    \label{E:constant-function}
    \approx
    \frac{A^{-1}(r^{-n})}{r^{n}\varphi(r)}
   \int_{0}^{r^n}\varrho^{\frac{s}{n}}\,d\varrho 
   \qquad \text{for $r>0$.}
\end{align}

Fix $r>0$. Let $B$ be the ball centered at the origin  such that 
$|B|=2r^n$, and choose 
\begin{equation*}
    f=A^{-1}\left(\tfrac{2}{|B|}\right)\chi_{(0,\frac{|B|}{2})}
\end{equation*}
in \eqref{uf}.
Observe that $f$ is non-increasing, $\operatorname{supp} f\subset(0,|B|)$, and $\|f\|_{L^A(0,\infty)}=1$. Thus, by \eqref{12}  and \eqref{15},
\begin{align*}
    A^{-1}( {r^{-n}})
    \int_{0}^{r^n} \varrho^{\frac{s}{n}}\,d\varrho
    &\le
    \sup\left\{\int_{0}^{\frac{|B|}{2}}
    f(\varrho)\, \varrho^{\frac{s}{n}}\,d\varrho:
    \|f\|_{L^A(0,\infty)}\le1,
    \ \operatorname{supp} f\subset(0,|B|),\ f\downarrow\right\}
        \\
    &\lesssim
    \sup\left\{\int_{B}|u_f-(u_f)_B|\,dx:
    \|f\|_{L^A(0,\infty)}\le1,
    \ \operatorname{supp} f\subset(0,|B|),\ f\downarrow\right\}
        \\
    &\lesssim
    \sup\left\{\int_{B}|u-u_B|\,dx:
    |u|_{s,A;\rn}\le1,
    \ \operatorname{supp} u\subset B\right\}.
\end{align*}
Making use of this estimate in~\eqref{E:constant-function}, we obtain, via ~\eqref{2v}:
\begin{align*}
    \lim_{r\to0_+}
    \frac{\varphi_{s,A}(r)}{\varphi(r)}
    &\lesssim
    \lim_{r\to0_+}
    \frac{1}{r^n\varphi(r)} 
    \sup\left\{\int_{B}|u-u_B|\,dx:
    |u|_{s,A;\rn}\le1,
    \ \operatorname{supp} u\subset B\right\}
        \\
    &\lesssim
    \lim_{r\to0_+}
    \sup_{|u|_{s,A;\rn}\le1}
    \sup_{B\subset\rn,|B|\le 2r^n}
    \frac{1}{\varphi((|B|/2)^{\frac{1}{n}})}
    \medint_{B}|u-u_B|\,dx
    =0.
\end{align*}
Hence, \eqref{limit} follows.
\end{proof}

\section{Case $s>1$, $k=0$}\label{sec4}

Here, we are concerned with  embeddings as in \eqref{emb-intro2}
 for $k=0$. We begin by detecting the admissible smoothness parameters $s$ and Young functions $A$ for which such an embedding is possible.

\begin{theorem}[\bf Admissible $s$ and $A$ for $k=0$]\label{teo3}
Let $s\in (1, \infty)\setminus \N$ and let $A$ be a Young function.  Assume that the   embedding
\begin{equation}\label{july1}
    V_{d,1}^{s,A}(\rn) 
    \to  
    \mathcal{L}^{\varphi} (\rn)
\end{equation}
holds for some  admissible function $\varphi$.  Then, 
\begin{equation}\label{E:s-plus-1}
    s\in (1, n+1)\setminus \N,
\end{equation}
and
\begin{equation}\label{intzero}
    \int_0 
    \left (\frac t{A(t)}\right)
    ^{\frac {s-1}{n+1- s}} \; dt <\infty\,.
\end{equation}
\end{theorem}

Given $s\in (1, n+1) \setminus \N$ and a Young function  $A$ satisfying the condition \eqref{intzero}, the optimal Campanato target space for embeddings of 
 $V_{d,1}^{s,A}(\rn)$ is built upon the function $\psi_{s,A} \colon  (0, \infty) \to [0, \infty)$ defined as follows. Let $F$ be the Young function given by
\begin{equation}\label{F}
F(t)=  t^{\frac {n}{n-(s-1)}}\, \int_0^t \frac{\widetilde{A} (\tau)}{\tau^{1+\frac {n}{n-(s-1)}}} \; d \tau \qquad \hbox{for}\;\; t\geq 0\,.
\end{equation}
Then, 
 the function 
$\psi_{s,A} : (0, \infty) \to [0, \infty)$ is defined as
\begin{equation}\label{101}
    \psi_{s,A}(r) 
    = \frac{1}{r^{n-s}
    F^{-1}\left(\frac 1{r^n}\right)} \qquad \hbox{for}\;\; r>0\,,
\end{equation}
where $F^{-1}$ denotes the right-continuous inverse of   $F$. 
Note that $F$ is well defined,  since assumption \eqref{intzero} is equivalent to the convergence of the integral on the right-hand side of equation \eqref{F} -- see \cite[Lemma 2.3]{cianchi-ibero}.
\\
 The alternative formula
\begin{equation}\label{127}
    \psi_{s,A} (r) \approx r^s \, {\widetilde{F}}^{-1}\left(\frac 1{r^n} \right)  \qquad \hbox{for}\;\; r>0
\end{equation}
holds, thanks to equation \eqref{AAtilde}.

\begin{remark}
 As shown in \cite[Proposition~2.2]{ACPS_modulus}, the function $A$ always dominates $\widetilde F$ globally. Moreover, if it is separated from the power function $t^{\frac n{s-1}}$ in the sense of the \emph{Matuszewska-Orlicz indices}, then $A$ and $\widetilde F$ are equivalent. Specifically, if
\begin{equation}\label{E:index-infty}
        i_{\infty}(A)>\frac{n}{s-1},
    \end{equation}
    then $\widetilde{F}\approx A$ near infinity and
    \begin{equation}\label{E:varrho-infty}
        \psi_{s,A}(r)\approx r^sA^{-1}\left(r^{-n}\right)
        \qquad\text{for $r\in(0,1)$.}
    \end{equation}
  Moreover, if   \begin{equation}\label{E:index-zero}
        i_{0}(A)>\frac{n}{s-1},
    \end{equation}
    then $\widetilde{F}\approx A$ near zero and
    \begin{equation}\label{E:varrho-zero}
        \psi_{s,A}(r)\approx r^sA^{-1}\left(r^{-n}\right)
        \qquad\text{for $r\in(1,\infty)$}.
    \end{equation}
\end{remark}

We are now in a position to state the optimal embedding theorem for $V_{d,1}^{s,A}(\rn)$ into Campanato spaces. A criterion for embeddings into vanishing spaces is also given.  
Analogously to~\eqref{E:vanishing}, we  say that $$V^{s,A}_{d,1}(\rn) \to V\!\mathcal{L}^{\varphi}(\rn) \quad \text {uniformly}$$
if $V^{s,A}_{d,1}(\rn) \to \mathcal L^\varphi (\rn)$ and 
\begin{equation}\label{E:vanishing-higher}
    \lim_{r\to0^+}
\sup_{ u\in V^{s,A}_{d,1}(\rn),|\nabla^{[s]}u|_{\{s\}, A, \rn} \leq 1 } 
    \left( 
    \sup_{B\subset \rn, |B| \leq r} \; \frac 1{\varphi (|B|^{\frac{1}{n}})}\medint _B |u - u_B| \; dx
    \right)=0.
\end{equation}

\begin{theorem}[\bf Optimal embeddings into Campanato spaces for $s>1$ and $k=0$]\label{teo2}
Assume that $s\in (1, n+1)  \setminus \N$. Let $A$ be a Young function fulfilling  \eqref{intzero} and let $\psi_{s,A}$ be the function defined by \eqref{101}.
\\ (i) The embedding
\begin{equation}\label{100}
 {V^{s, A}_{d,1}(\rn)} \to \mathcal{L}^{\psi_{s,A}} (\rn)
\end{equation}
holds, and there exists a positive constant $c=c(n,s)$ such that
\begin{equation}\label{102}
|u|_{\mathcal{L}^{\psi_{s,A}} (\rn)} \leq c \, \big|\nabla^{[s]} u \big|_{\{s\}, A, \rn}
\end{equation}
for every $u\in {V^{s, A}_{d,1}(\rn)}$.
 Moreover, $\mathcal{L}^{\psi_{s,A}} (\rn)$ is the optimal  Campanato target space in \eqref{100} and \eqref{102}.
\\ (ii)
Let $\varphi$ be an admissible function. Then,
\begin{equation}\label{100v}
V^{s, A}_{d,1}(\rn) \to V\!\mathcal{L}^{\varphi}(\rn)\qquad \text{uniformly}
\end{equation}
if and only if
\begin{equation}\label{limit1'}
\psi_{s,A}(r)\leq c \, \varphi (r) \qquad \text{for $r>0$,}
\end{equation}
for some positive constant $c$
and
\begin{equation}\label{limit'}
\lim_{r \to 0^+} \frac{\psi_{s,A}(r)}{\varphi (r)} =0.
\end{equation}
\end{theorem}

Considerations parallel to those following Theorem \ref{teo1} on the non-effectiveness of the use of embeddings into Campanato type spaces to deduce embeddings into spaces of uniformly continuous functions also apply in the case of higher-order fractional Orlicz-Sobolev spaces. 
\\
Indeed, \cite[Corollaries 1 and 2]{Spanne} imply that the space $\mathcal{L}^{\psi_{s, A}} (\rn)$ is embedded into some space of uniformly continuous functions  
if and only if
\begin{equation}\label{E:spanne-higher-2}
    \int_0 \frac{\psi_{s,A}(r)}{r}\; dr < \infty.
\end{equation}
  The  relevant embedding reads:
\begin{equation}\label{E:spanne-higher-1}
 \mathcal{L}^{\psi_{s, A}} (\rn) \to C^{\varphi(\cdot)}(\rn),
\end{equation}
where
\begin{equation}\label{E:spanne-higher-3}
    \varphi (r)\approx\int_0^{r} \frac{\psi_{s,A}(\varrho)}{\varrho}\; d\varrho\qquad\text{for $r>0$.}
\end{equation}
 The condition \eqref{E:spanne-higher-2} amounts to
$$\int_0r^{s-1}\widetilde{F}^{-1}(r^{-n})\, dr <\infty.$$
A change of variables shows that the latter is equivalent to 
\begin{align}
    \label{2024-10}
    \int^\infty \frac{\widetilde{F}^{-1}(t)}{t^{1+ \frac sn}}\, dt <\infty.
\end{align}
Altogether, combining \eqref{100} with~\eqref{E:spanne-higher-1}, one can conclude  that
\begin{equation}
\label{E:spanne-higher-1bis}
    V^{s,A}_{d,1} (\rn)
    \to 
    C^{\varphi(\cdot)}(\rn)
\end{equation}
with $\varphi$ given by \eqref{E:spanne-higher-3}.
\\  On the other hand,
 the condition \eqref{2024-10} can be shown to imply (and to be possibly stronger than)
the condition 
\begin{equation}\label{2024-12}
    \int^{\infty} 
    \bigg (\frac t{\widetilde{F}(t)}\bigg)^{\frac s{n-s}} \; dt <\infty.
\end{equation}
Since $\widetilde{F} \lesssim A$, the condition \eqref{2024-12} in turn implies (and is again possibly stronger than) \eqref{E:weaker}. The latter  is the optimal condition for an embedding of the form \eqref{E:spanne-higher-1bis} to hold for some modulus of continuity $\varphi$ -- see \cite{ACPS_modulus}. Moreover, there exist $s$ and $A$ for which the conditions \eqref{2024-10} and \eqref{E:weaker} hold simultaneously, and the embedding \eqref{E:spanne-higher-1bis} is weaker than the optimal one established in \cite{ACPS_modulus}. This can be verified by considering again Young functions $A$ as in \eqref{2024-50}.

An application of Theorem \ref{teo2} yields the following characterization of embeddings into $\BMO(\rn)$ and $\VMO(\rn)$.
 
\begin{corollary}[\bf Embeddings into $\BMO$ for $s>1$]\label{C:bmo-higher1}
Let $s\in (1,n+1)\setminus \N$ and let $A$ be a Young function.
\\ (i) The embedding
\begin{equation}\label{2024-51}
V^{s,A}_{d,1}(\rn) \to \BMO(\rn)
\end{equation}
holds if and only if either $s<n$ and
\begin{equation}\label{2024-52}
\int_0^t \frac{\widetilde{A} (\tau)}{\tau^{1+\frac {n}{n-(s-1)}}} \; d \tau \leq c \, t^{\frac n{(n-s)(n-(s-1))}}\quad\text{for $t>0$,}
\end{equation}
for some positive constant $c$,
or $s=n$ and 
\begin{align}
    \label{2024-56}
    A(t)\geq  c\, t \qquad\text{for $t\in [0, t_0]$,}
\end{align}
for some positive constants $c$ and $t_0$.
\\ (ii) We have that
\begin{equation}\label{2024-53}
V^{s,A}_{d,1}(\rn) \to \VMO(\rn) \qquad \text{uniformly}
\end{equation}
if and only if either $s<n$, the condition \eqref{2024-52} holds, and
\begin{equation}\label{2024-54}
\lim_{t \to \infty}  
t^{-\frac n{(n-s)(n-(s-1))}}\int_0^t \frac{\widetilde{A} (\tau)}{\tau^{1+\frac {n}{n-(s-1)}}} \; d \tau
= 0,
\end{equation}
or $s=n$ and the condition \eqref{2024-56} holds.
\end{corollary}

\begin{proof}
[Proof of Theorem \ref{teo3}]  
Assume 
that embedding \eqref{july1} holds for some $s\in (1, \infty) \setminus \N$.
Let $\xi \in C^\infty_0(\rn)$  be  a nonnegative function such that $\nabla \xi (0)\neq 0$. For each $j \in \N$,   consider the function $u_j \colon  \rn \to \mathbb  R$ defined as
\begin{equation}\label{100a}
u_j (x)= j^{s-n} \xi \Big(\frac x j\Big) \qquad \hbox{for}\;\; x\in \rn\,.
\end{equation}
Since $u_j \in C^\infty_0(\rn)$, we have that
$$|\{ |\nabla ^k u_j|> t\}| < \infty \qquad \text{ for  $t>0$,}$$
 for  $k=0, 1, \dots, [s]$.
\\
We claim that there exists a constant $c$, independent of $j$, such that
\begin{equation}\label{may30}
|\nabla ^{[s]}u_j|_{\{s\}, A, \rn} \leq c.
\end{equation}
To verify this claim, observe that
\begin{equation}\label{101a}
 \nabla ^{[ s] }u_j (x) = j^{s- [ s] -n}  \nabla ^{[ s] } \xi  \Big(\frac x j\Big) = j^{\{s\} -n}   \nabla ^{[ s] } \xi  \Big(\frac x j\Big) \qquad \text{for $x \in \rn$.}
\end{equation}
Therefore,
\begin{equation}\label{102a}
 \frac{\Big|\nabla ^{[ s] }u_j(x) - \nabla ^{[ s] }u_j(y)\Big|}{|x-y|^{\{s\}}} =  \frac{\Big|\nabla ^{[ s] }\xi\big(\frac x j\big) - \nabla ^{[ s] }\xi\big(\frac y j\big)\Big|}{\big|\frac{x-y}j \big|^{\{s\}}}\, j^{-n} \qquad \text{for $x,y \in \rn$, with $x \neq y$.}
\end{equation}
Since $\xi$ is smooth and has  bounded support, and $j \geq 1$, the right-hand side of equation \eqref{102a} is pointwise bounded by a constant $t_0$ independent of $j$.
\\
Next, since   $A$ is a Young function,  there  exists  a positive constant $c$ 
such that $A(t)\leq c\, t$ for $t\in [0, t_0]$.  Hence,
\begin{align}\label{103}
|\nabla ^{[ s] }u_j|_{\{s\}, A, \rn} &\leq c \, \int_{\rn} \int_{\rn} \frac{\big|\nabla ^{[ s] }u_j(x) - \nabla ^{[ s] }u_j(y)\big|}{|x-y|^{\{s\}}}\; \frac{dx \, dy}{|x-y|^n}
\\
& = c\, j^{-2n} \, \int_{\rn} \int_{\rn} \frac{\Big|\nabla ^{[ s] }\xi\big(\frac x j\big) - \nabla ^{[ s] }\xi\big(\frac y j\big)\Big|}{\big|\frac{x-y}j \big|^{\{s\}}}\; \frac{dx \, dy}{|\frac {x-y}j|^n}\nonumber
\\
& = c\,  \, \int_{\rn} \int_{\rn} \frac{\big|\nabla ^{[ s] }\xi(x) - \nabla ^{[ s]}\xi(y)\big|}{|x-y|^{\{s\}}}\; \frac{dx \, dy}{|x-y|^n}\,. \nonumber
\end{align}
The assumption that $\xi \in C^\infty_0(\rn)$  ensures that the integral on
the rightmost-hand side  of equation \eqref{103} is convergent. Hence, the inequality \eqref{may30} follows.
\\ The embedding \eqref{july1} and the inequality \eqref{may30} imply that there exists a constant $c$ such that
\begin{equation}\label{july2}
|u_j|_{\mathcal L^{\varphi}(\rn)} \leq c
\end{equation}
for  $j \in \N$. For $r>0$, we will denote by $B_r(0)$ the ball centered at $0$ and having radius $r$. From the definition of $u_j$ and a change of  variables, we obtain that
\begin{align}\label{july3}
    |u_j|_{\mathcal L^{\varphi}(\rn)} 
    \geq    
    \frac 1
    {\varphi (\omega_n^{\frac{1}{n}})} 
    \medint_{B_1(0)}|u_j-(u_j)_{B_1(0)}|\, dx
    & =
    \frac {j^{s-n}}
    {\varphi (\omega_n^{\frac{1}{n}})} \medint_{B_{\frac 1j}(0)}|\xi-\xi_{B_{\frac 1j}(0)}|\, dx
\end{align}
for  $j \in \N$. 
Observe that
\begin{equation}\label{july4}
\liminf_{j \to \infty}\,  j \medint_{B_{\frac 1j}(0)}|\xi-\xi_{B_{\frac 1j}(0)}|\, dx >0.
\end{equation}
Indeed, inasmuch as
\begin{equation}
    \label{dec1}
    \xi(x) = \xi(0) + x \cdot \nabla \xi (0) + \mathcal O (x^2) \qquad \text{as $x\to 0$,}
\end{equation}
one has that
\begin{equation}
    \label{dec2}
\xi_{B_{\frac 1j}(0)} = \xi (0) + \medint_{B_{\frac 1j}(0)}x \cdot \nabla \xi (0) dx + \mathcal O (j^{-2})
= \xi (0) + \mathcal O (j^{-2})
\qquad \text{as $j\to \infty$.}
    \end{equation}
Thus, since we are assuming that $\nabla \xi (0)\neq 0$, there exists a positive constant $c$ such that
    \begin{align}
    \label{dec3}
\liminf_{j \to \infty}\,  j \medint_{B_{\frac 1j}(0)}|\xi-\xi_{B_{\frac 1j}(0)}|\, dx
& = \liminf_{j \to \infty}\,  j \medint_{B_{\frac 1j}(0)}|x \cdot \nabla \xi (0) + \mathcal O (x^2)  + \mathcal O (j^{-2})|\, dx
\\ \nonumber  & \geq \liminf_{j \to \infty}\,  \left(j \medint_{B_{\frac 1j}(0)}|x \cdot \nabla \xi (0)|\, dx - \mathcal O (j^{-1})\right)
\\ \nonumber  &  \geq \lim_{j \to \infty} \left(c - \mathcal O (j^{-1})\right)
=c,
\end{align}
whence equation \eqref{july4} follows.
From equations \eqref{july3} and \eqref{july4}, we deduce that there exists a positive constant $c$ such that
\begin{equation}\label{july5}
|u_j|_{\mathcal L^{\varphi}(\rn)} \geq   c\, j^{s-n-1}
\end{equation}
for sufficiently large $j$. Since the latter inequality can only hold if $s<n+1$, the assertion \eqref{E:s-plus-1} follows.
\\ 
The fact that an  embedding of the form \eqref{july1}  can only hold under the assumption  \eqref{intzero} is a byproduct  of the
proof of Theorem \ref{teo2} to be presented below. More precisely, the argument that establishes the inequality \eqref{E:pump} tells us that if \eqref{july1} holds for some admissible function $\varphi$, then
\begin{equation}\label{2024-58}
\frac{|B|^{\frac{1}{n}}}
{\varphi(|B|^{\frac{1}{n}})}
        \left\|
        r^{-1+\frac{s-1}{n}}
        \chi_{(|B|,\infty)}(r)
        \right\|_{L^{\widetilde A}(0,\infty)}\le c
\end{equation}
 for every ball $B\subset\rn$ and some constant 
$c$ depending only on $s$, $A$, and $n$.
A standard computation shows that the finiteness of the Luxemburg norm in \eqref{2024-58} for some $B$ is equivalent to \eqref{intzero}.
More precisely, the finiteness of the Luxemburg norm 
in \eqref{2024-58} for some $B$ is equivalent to the existence of some $\lambda>0$ such that
\begin{equation*}
    \int_{|B|}^{\infty}
    \widetilde A\left(\frac{r^{-1+\frac{s-1}{n}}}{\lambda}\right)\,dr<\infty.
\end{equation*}
A change of variables shows that this condition is in turn equivalent to
\begin{equation*}
\int_{0}\frac{\widetilde A(y)}{y^{\frac{n}{n-s+1}+1}}\,dy < \infty.
\end{equation*}
By \eqref{equiv-cond}, the latter condition is  equivalent to \eqref{intzero}. 
\end{proof}

The following technical lemmas are needed in the proof of  Theorem \ref{teo2}, as well as in that of Theorem \ref{teo_k} in the next section.

\begin{lemma}\label{L:lemma-for-ri-norms}
    Let $X(0,\infty)$ be a~rearrangement-invariant space. Assume that $\alpha\in(0,\infty)$ and $\beta\in(-1,\infty)$ are such that $\alpha+\beta\ge0$. Then,
    \begin{equation}\label{E:lemma-for-ri-norms}
        \left\|\varrho^{\alpha}\left((\cdot)^{\beta}\chi_{(0,r)}(\cdot)\right)^{**}(\varrho)
        \right\|_{X(0,\infty)}
        \approx
        r^{\beta+1}\left\|\varrho^{\alpha-1}\chi_{(r,\infty)}(\varrho)\right\|_{X(0,\infty)} \qquad \text{for $r>0$}
    \end{equation}
    with equivalence constants depending only on $\beta$.
\end{lemma}

\begin{proof} 
    Fix $r\in(0,\infty)$. Suppose first that $\beta\ge 0$. Then,
    \begin{equation}\label{E:lemma-for-ri-norms-4}
        \left((\,\cdot\,)^{\beta}\chi_{(0,r)}(\,\cdot\,)\right)^{*}(\varrho)
        =(r-\varrho)^{\beta}\chi_{(0,r)}(\varrho)
    \end{equation}
   and
    \begin{align}
    \left((\,\cdot\,)^{\beta}\chi_{(0,r)}(\,\cdot\,)\right)^{**}(\varrho)
        &=
        \frac{1}{\beta+1}\varrho^{-1}\left[r^{\beta+1}-(r-\varrho)^{\beta+1}\right]\chi_{(0,r)}(\varrho)
        \label{E:lemma-for-ri-norms-1}
            \\
     & +
        \frac{1}{\beta+1}r^{\beta+1}\varrho^{-1}\chi_{(r,\infty)}(\varrho)
        \qquad\text{for $\varrho >0$.}
        \nonumber
    \end{align}
    Owing to the convexity of the function $\varrho\mapsto \varrho^{\beta+1}$ on the interval $(0,r)$, one has
    \begin{equation}\label{E:lemma-for-ri-norms-5}
        r^{\beta+1}-(r-\varrho)^{\beta+1}
        \le (\beta+1)\,r^{\beta}\varrho
        \qquad\text{for $\varrho\in(0,r)$.}
    \end{equation}
    Combining~\eqref{E:lemma-for-ri-norms-1} with~\eqref{E:lemma-for-ri-norms-5}, we get
    \begin{align}
    \left((\,\cdot\,)^{\beta}\chi_{(0,r)}(\,\cdot\,)\right)^{**}(\varrho)
        &\le
        r^{\beta}\chi_{(0,r)}(\varrho)
        \label{E:lemma-for-ri-norms-1-bis}+
        \frac{1}{\beta+1}r^{\beta+1}\varrho^{-1}\chi_{(r,\infty)}(\varrho)
        \qquad\text{for $\varrho >0$.}
    \end{align}
    Thus,
    \begin{align}
    \label{E:lemma-for-ri-norms-6}
            \left\|\varrho^{\alpha}\left((\,\cdot\,)^{\beta}\chi_{(0,r)}(\,\cdot\,)\right)^{**}(\varrho)
            \right\|_{X(0,\infty)}
            &\le
            r^{\beta}\left\|\varrho^{\alpha}\chi_{(0,r)}(\varrho)\right\|_{X(0,\infty)}
                +
            \frac{1}{\beta+1}r^{\beta+1}\left\|\varrho^{\alpha-1}\chi_{(r,\infty)}(\varrho)\right\|_{X(0,\infty)}.
    \end{align}
    Since the function $\varrho\mapsto \varrho^{\alpha}$ is increasing on $(0,\infty)$ and $X(0,\infty)$ is a rearrangement-invariant space, we have
    \begin{align}
        \label{E:lemma-for-ri-norms-3}
        \left\|\varrho^{\alpha}\chi_{(0,r)}(\varrho)\right\|_{X(0,\infty)}
        & \le
        \left\|\varrho^{\alpha}\chi_{(r,2r)}(\varrho)\right\|_{X(0,\infty)}
            \\  
         & \le  2r
        \left\|\varrho^{\alpha-1}\chi_{(r,2r)}(\varrho)\right\|_{X(0,\infty)}
        \le 2r
        \left\|\varrho^{\alpha-1}\chi_{(r,\infty)}(\varrho)\right\|_{X(0,\infty)}.
            \nonumber
    \end{align}    
    Coupling~\eqref{E:lemma-for-ri-norms-3} with~\eqref{E:lemma-for-ri-norms-6} yields
    \begin{equation}
    \label{E:lemma-for-ri-norms-8}
             \left\|\varrho^{\alpha}\left((\,\cdot\,)^{\beta}\chi_{(0,r)}(\,\cdot\,)\right)^{**}(\varrho)
            \right\|_{X(0,\infty)}
            \le
            \left(
            2+\frac{1}{\beta+1}
            \right)
            r^{\beta+1}\left\|\varrho^{\alpha-1}\chi_{(r,\infty)}(\varrho)\right\|_{X(0,\infty)}.
    \end{equation}
    Conversely, neglecting the first addend on the right-hand side of~\eqref{E:lemma-for-ri-norms-1} and using the lattice property of a~rearrangement-invariant space, we get 
    \begin{equation}
    \label{E:lemma-for-ri-norms-9}
             \left\|\varrho^{\alpha}\left((\,\cdot\,)^{\beta}\chi_{(0,r)}(\,\cdot\,)\right)^{**}(\varrho)
            \right\|_{X(0,\infty)}
            \ge
            \frac{1}{\beta+1}
            r^{\beta+1}\left\|\varrho^{\alpha-1}\chi_{(r,\infty)}(\varrho)\right\|_{X(0,\infty)},
    \end{equation}
    and equation~\eqref{E:lemma-for-ri-norms} follows from ~\eqref{E:lemma-for-ri-norms-8} and~\eqref{E:lemma-for-ri-norms-9}.
        \\
    Now let $\beta\in(-1,0)$. Then,
    \begin{equation}\label{E:lemma-for-ri-norms-10}
        \left((\,\cdot\,)^{\beta}\chi_{(0,r)}(\,\cdot\,)\right)^{*}(\varrho)
        =\varrho^{\beta}\chi_{(0,r)}(\varrho) \qquad\text{for $\varrho >0$}
    \end{equation}
    and 
    \begin{equation}\label{E:lemma-for-ri-norms-11}
        \left((\,\cdot\,)^{\beta}\chi_{(0,r)}(\,\cdot\,)\right)^{**}(\varrho)
        =
        \frac{1}{\beta+1}
        \varrho^{\beta}\chi_{(0,r)}(\varrho)
        +
        \frac{1}{\beta+1}
        r^{\beta+1}{\varrho}^{-1}\chi_{(r,\infty)}(\varrho)
        \qquad\text{for $\varrho >0$.}
    \end{equation}
    Therefore,
    \begin{align}
        \label{E:lemma-for-ri-norms-12}
        &\left\|\varrho^{\alpha}\left((\,\cdot\,)^{\beta}\chi_{(0,r)}(\,\cdot\,)\right)^{**}(\varrho)
        \right\|_{X(0,\infty)}
            \\
        &\qquad\le
        \frac{1}{\beta+1}
        \left\|
        \varrho^{\alpha+\beta}\chi_{(0,r)}(\varrho)
        \right\|_{X(0,\infty)}
        +
        \frac{1}{\beta+1}
        r^{\beta+1}
        \left\|{\varrho}^{\alpha-1}\chi_{(r,\infty)}(\varrho)
        \right\|_{X(0,\infty)}
        \qquad\text{for $\varrho >0$.}\nonumber
    \end{align}
    Since $\alpha+\beta\ge0$, the function $\varrho\mapsto \varrho^{\alpha+\beta}$ is increasing on the interval $(0,r)$. Thus, using also the rearrangement-invariance of $X$ and the fact that the function $\varrho\mapsto \varrho^{\beta+1}$ is increasing on $[r,2r]$, we have
    \begin{align}
        \label{E:lemma-for-ri-norms-13}
        \left\|
        \varrho^{\alpha+\beta}
        \chi_{(0,r)}(\varrho)
        \right\|_{X(0,\infty)}
        &\le
        \left\|
        \varrho^{\alpha+\beta}
        \chi_{(r,2r)}(\varrho)
        \right\|_{X(0,\infty)}
            \le(2r)^{\beta+1}
        \left\|
        \varrho^{\alpha-1}
        \chi_{(r,2r)}(\varrho)
        \right\|_{X(0,\infty)}
            \\
        &\le
        2^{\beta+1}r^{\beta+1}
        \left\|
        \varrho^{\alpha-1}
        \chi_{(r,\infty)}(\varrho)
        \right\|_{X(0,\infty)}.
            \nonumber
    \end{align}
    Plugging this into~\eqref{E:lemma-for-ri-norms-12}, one arrives at
    \begin{align}
        \label{E:lemma-for-ri-norms-14}
        \left\|\varrho^{\alpha}\left((\,\cdot\,)^{\beta}\chi_{(0,r)}(\,\cdot\,)\right)^{**}(\varrho)
        \right\|_{X(0,\infty)}
        &\le
        \frac{2^{\beta+1}+1}{\beta+1}
        r^{\beta+1}
        \left\|
        \varrho^{\alpha-1}
        \chi_{(r,\infty)}(\varrho)
        \right\|_{X(0,\infty)}.
     \end{align}
     Conversely, neglecting the first addend in~\eqref{E:lemma-for-ri-norms-11}, we get~\eqref{E:lemma-for-ri-norms-9} once again (with the same constant). This establishes~\eqref{E:lemma-for-ri-norms} for $\beta\in(-1,0)$ and finishes the proof.
\end{proof}

\begin{lemma}\label{lemma-old}
Let $\gamma >0$, let $A$ be a Young function and let $R>0$. Assume that $g$ is a nonnegative function in $L^A(0, \infty)$, vanishing on $(0, R)$, and 
 non-increasing in $(R, \infty)$. Then, there exists a  nonnegative  non-increasing function $f\in L^A(0, \infty)$ such that
\begin{equation*}
\|f\|_{L^A(0, \infty)}\leq c \|g\|_{L^A(0, \infty)}
\end{equation*}
for some  absolute constant $c$, and
\begin{equation*}
\int_R^\infty g(r) \, r^{-\gamma} \; dr \approx \int_R^\infty f(r) \, r^{-\gamma} \; dr\,,
\end{equation*}
up to multiplicative constants depending only on $\gamma$.
\end{lemma}

 Lemma \ref{lemma-old} is stated in 
 \cite[Lemma~5.2]{ACPS_modulus} for  $\gamma > 1$. An inspection of  its proof reveals that, in fact, it applies for every $\gamma >0$.

\begin{proof}[Proof of Theorem \ref{teo2}]
      Throughout this proof, the constants appearing in the estimates and in the relations  $\lq\lq \lesssim "$ and $\lq\lq \approx "$ depend only on $n$ and $s$.
    \\ Part (i). 
    An application of Theorem~\ref{thm_Z} with $s-1$ in place of $s$ (which is possible thanks to our assumptions $s\in(1,n+1)$ and~\eqref{intzero})
     guarantees that 
    \begin{equation}\label{emb-s-1}
        \| v \|_{ X_{s-1}(\rn)} \leq c \,| \nabla^{[s]-1} v|_{\{s\}, A , \rn}
    \end{equation}
    for some constant $c$ and every $v\in V_{d,0}^{s,A}(\rn)$.  If $u\in V^{s,A}_{d,1}(\rn)$, then,
    by the very definition of $V^{s,A}_{d,1}(\rn)$, each first-order partial derivative of $u$ belongs to $V^{s-1,A}_{d,0}(\rn)$, and
    \begin{equation}\label{103'}
    \big| \nabla^{[s]-1} (\nabla  u)\big|_{ \{s\}, A, \rn}  = \big|\nabla^{[s]} u\big|_{ \{s\}, A, \rn}.
    \end{equation}
    Thus, applying~\eqref{emb-s-1} with $v$ replaced with each component of $\nabla u$ and coupling the resultant inequality  with~\eqref{103'} yield
    \begin{equation}\label{104'}
    \| \nabla u \|_{ X_{s-1}(\rn)} \leq c| \nabla^{[s]} u|_{\{s\}, A,\rn}
    \end{equation}
    for every $u\in V_{d,1}^{s,A}(\rn)$. Thus, equations ~\eqref{100} and~\eqref{102} will follow if we show that
    \begin{equation}\label{E:may-9-one}
        |u|_{\mathcal{L}^{\psi_{s,A}}(\rn)}\lesssim\|\nabla u\|_{X_{s-1}(\rn)}
    \end{equation}
    for every $u\in V^1X_{s-1}(\rn)$.
    Thanks to Theorem~A, the proof of ~\eqref{E:may-9-one} reduces to verifying that
    \begin{equation}\label{E:may-9-two}
        \psi_{s,A}(r)
            \approx
        r^{-n}
        \|\varrho^{\frac{1}{n}}\chi_{(0,r^n)}(\varrho)\|_{X_{s-1}'(0,\infty)}\qquad\text{for $r\in(0,\infty)$.}
    \end{equation}
    Fix $r>0$. By the definition of the norm $\|\cdot\|_{X_{s-1}(0,\infty)}$, 
    \begin{equation}\label{E:may-9-three}
        \|\varrho^{\frac{1}{n}}\chi_{(0,r^n)}(\varrho)\|_{X_{s-1}'(0,\infty)}
            =
        \left\|
        \varrho^{\frac{s-1}{n}}
        \left[
        (\,\cdot\,)^{\frac{1}{n}}
        \chi_{(0,r^n)}(\,\cdot\,)
        \right]^{**}(\varrho)
        \right\|_{L^{\tilde A}(0,\infty)}.
    \end{equation}
    From Lemma~\ref{L:lemma-for-ri-norms} applied with $\alpha=\frac{s-1}{n}$, $\beta=\frac{1}{n}$ and $r$ replaced by $r^n$, we deduce that
    \begin{equation}\label{E:may-9-four}
        \left\|
        \varrho^{\frac{s-1}{n}}
        \left[
        (\,\cdot\,)^{\frac{1}{n}}
        \chi_{(0,r^n)}(\,\cdot\,)
        \right]^{**}(\varrho)
        \right\|_{L^{\tilde A}(0,\infty)}
            \approx
        r^{n+1}
        \left\|
        \varrho^{\frac{s-1}{n}-1}
        \chi_{(r^n,\infty)}(\varrho)
        \right\|_{L^{\tilde A}(0,\infty)},
    \end{equation}
    with equivalence constants depending only on $n$.
    We have that
    \begin{equation}\label{aug100}
        \big\|
        \varrho^{\frac {s-1}n -1}
        \chi_{(r^n, \infty)}(\varrho) 
        \big \|_{L^{\widetilde{A}} (0, \infty)}
            =
        \inf\left\{ \lambda >0: 
        \int_{r^n}^{\infty}\widetilde{A}\left (\frac{\varrho^{-1+ 
        \frac{s-1}{n}}}{\lambda}\right)\,d\varrho
        \leq 1 \right\} 
        \quad \text{for $r>0$.}
    \end{equation}
    Computations   show that 
    \begin{equation}\label{E:may-9-five}
       \int_{r^n}^{\infty}\widetilde{A}\left (\frac{\varrho^{-1+ 
        \frac{s-1}{n}}}{\lambda}\right)\,d\varrho
            =
        \frac{n}{n+1-s}\, r^n
        F\left(
        \frac{r^{s-n-1}}{\lambda}
        \right),
    \end{equation}
    see~\eqref{F}.
    Hence, by~\eqref{aug100} and~\eqref{E:may-9-five},
    \begin{equation}\label{E:may-9-five-bis}    
        \big\|
        \varrho^{\frac {s-1}n -1}
        \chi_{(r^n, \infty)}(\varrho) 
        \big \|_{L^{\widetilde{A}}(0,\infty)}       \approx
        \frac{r^{s-n-1}}{F^{-1}(r^{-n})}
        \qquad \text{for $r>0$.}
    \end{equation}
    By~\eqref{E:may-9-three}, \eqref{E:may-9-four}, ~\eqref{E:may-9-five-bis} and~\eqref{101}, we obtain that
    \begin{equation}\label{E:may-9-six}
        r^{-n}
        \|\varrho^{\frac{1}{n}}\chi_{(0,r^n)}(\varrho)\|_{X_{s-1}'(0,\infty)}
            \approx
        r\,\big\|
        \varrho^{\frac {s-1}n -1}
        \chi_{(r^n, \infty)}(\varrho) 
        \big \|_{L^{\widetilde{A}}(0,\infty)}   
            \approx
        \frac{r^{s-n}}{F^{-1}(r^{-n})}
            =
        \psi_{s,A}(r).
    \end{equation}
   Equation ~\eqref{E:may-9-two} is thus established.  

    It remains to show that 
    $\mathcal{L}^{\psi_{s, A}}(\rn)$  is optimal
    among all Campanato type target spaces.
    Assume that $\varphi$ is an admissible function such that
    \begin{equation}\label{102-bis}
        |u|_{\mathcal{L}^{\varphi}(\rn)}       \leq
        c\, \big|\nabla^{[s]}u\big|_{\{s\},A,\rn}
    \end{equation}
    for $u \in V^{s,A}_{d,1}(\rn)$.
     Given $L>0$ and a nonnegative non-increasing  function $f\in L^A(0,\infty)$ with ${\rm supp} \,f \subset [0, L]$,  define the function $v_f\colon \rn \to \R$ as
      \begin{equation}\label{E:uf}
        v_f(x)=x_1
        \int_{\omega_n|x|^n}^\infty
        f(r) \, r^{-[s]-1+\frac{s-1}{n}}
        (r-\omega_n|x|^n)^{[s]}\,dr \qquad\text{for $x\in\rn$.}
      \end{equation}
       Fix any ball $B\subset\rn$ centered at $0$ and choose $L$ so that $L>|B|$. 
      By~\cite[Lemma~5.1]{ACPS_modulus},  
      \begin{equation}\label{128}
    |\nabla^{[s]}  v_f |_{\{s\}, A, \rn} \lesssim \|f\|_{L^A(0, \infty)}\,.
    \end{equation}
    Hence, $v_f \in V^{s,A}_{d,1}(\rn)$. Moreover,  $(v_f)_{B} =0$, inasmuch as $v_f$ is odd with respect to $x_1$.
    Therefore,
    \begin{equation}\label{129}
    \int_{B} |v_f(x)- (v_f)_{B}|\; dx =\int_{B} |v_f(x)|\; dx\,.
    \end{equation}
    Set
    \begin{equation*}
    E= \{x \in  B\colon 2\, |x_1|   \geq |x| \}\,.
    \end{equation*}
    The following chain relies upon  changes of variables and  elementary estimates:
    \begin{align}\label{130}
        \int_{B} |v_f (x)|\,d x  
            &=   
        \int_{B} |x_1| \int_{\omega_n|x|^n}^\infty f(r)\, r^{-[s]-1+\frac{s-1}{n}}
        (r-\omega_n|x|^n)^{[s]}\,dr\,dx
             \\
         &\geq
        \int_{E} \frac{|x|}{2} 
        \int_{\omega_n|x|^n}^\infty   
        f(r)\, r^{-[s]-1+\frac{s-1}{n}}
        (r-\omega_n|x|^n)^{[s]}\,dr\,dx
        \nonumber
            \\
        &\approx 
        \int_{B}|x| 
        \int_{\omega_n|x|^n}^\infty 
        f(r)\, r^{-[s]-1+\frac{s-1}{n}}
        (r-\omega_n|x|^n)^{[s]}\,dr\,dx\nonumber
            \\
        &\approx 
        \int_0^{|B|
        ^{\frac{1}{n}}}\varrho^n 
        \int_{\omega_n \varrho^n}^\infty   
        f(r)\, r^{-[s]-1+\frac{s-1}{n}}
        (r-\omega_n\rho^n)^{[s]}\,dr\,d\varrho\nonumber
            \\
        &\approx 
        \int_0^{|B|}
        \sigma^{\frac 1n}
        \int_{\sigma}^\infty
        f(r)\, r^{-[s]-1+\frac{s-1}{n}}
        (r-\sigma)^{[s]}\,dr\,d\sigma\nonumber
            \\
        &\ge
        \int_0^{\frac{|B|}{2}} 
        {\sigma}^{\frac 1n}  
        \int_{|B|}^\infty
        f(r)\, r^{-[s]-1+\frac{s-1}{n}}
        (r-\sigma)^{[s]}\,dr\,d\sigma\nonumber
          \\
        &\approx
        \int_0^{\frac{|B|}{2}}
        {\sigma}^{\frac 1n}
        \int_{|B|}^\infty
        f(r)\, r^{-[s]-1+\frac{s-1}{n}}r^{[s]}\,dr\,
        d\sigma\nonumber
            \\
        &\approx  
        |B|^{1+\frac{1}{n}}
        \int_{|B|}^{\infty}  
        f(r)\, r^{-1+\frac{s-1}{n}}\,dr.\nonumber
    \end{align}
    From \eqref{102-bis},
    \eqref{129}, \eqref{128} and~\eqref{130} we obtain that
    \begin{align}\label{E:dau}
        c &\ge
        \sup\left\{
        \frac{1}
        {|\nabla^{[s]}v_f|_{\{s\},A,\rn}}
        \frac{1}
        {\varphi(|B|^{\frac{1}{n}})}
        \medint_B
        |v_f|\,dx: 
        f\downarrow,\ 
        \operatorname{supp} f
        \subset[0,L]
        \right\}
            \\
        & \gtrsim
        \sup\left\{
        \frac{1}{\|f\|_{L^A(0,\infty)}}
        \frac{|B|^{\frac{1}{n}}}
        {\varphi(|B|^{\frac{1}{n}})}
        \int_{|B|}^{\infty}
        f(r)\, r^{-1+\frac{s-1}{n}}\,
        dr:
        f\downarrow,\ 
        \operatorname{supp} f
        \subset[0,L]\right\}
            \nonumber \\
        &
        = \sup\left\{
        \frac{1}{\|f\|_{L^A(0,\infty)}}
        \frac{|B|^{\frac{1}{n}}}
        {\varphi(|B|^{\frac{1}{n}})}
        \int_{|B|}^{\infty}
        f(r)\, r^{-1+\frac{s-1}{n}}\,
        dr:
        f\downarrow\right\},
            \nonumber
    \end{align}
where the equality holds owing to an approximation argument and the Fatou property of Luxemburg norms. 
\\
    Thanks to  Lemma \ref{lemma-old},
    the last supremum in \eqref{E:dau} does not  increase, up to a multiplicative constant, if the class of admissible functions $f$ is modified as to consist of 
    those $f$ which vanish on $(0,|B|)$ and do not increase on $(|B|,\infty)$. 
    Thus, one has that
\begin{equation}\label{E:tri}
        c \gtrsim
        \sup\left\{
        \frac{1}{\|f\|_{L^A(0,\infty)}}
        \frac{|B|^{\frac{1}{n}}}
        {\varphi(|B|^{\frac{1}{n}})}
        \int_{|B|}^{\infty}
        f(r)\, r^{-1+\frac{s-1}{n}}\,
        dr:
        f\downarrow\ \text{on $(|B|,\infty)$},\ 
         f=0 \, \text{in}\,   [0, |B|]\right\}.
    \end{equation}
    Hence, by ~\eqref{E:dual-monotone}, 
\begin{equation}\label{E:pump}
        c \gtrsim
        \frac{|B|^{\frac{1}{n}}}
        {\varphi(|B|^{\frac{1}{n}})}
        \left\|
        r^{-1+\frac{s-1}{n}}
        \chi_{(|B|,\infty)}(r)
        \right\|_{L^{\widetilde{A}}(0,\infty)}.
    \end{equation}
    Consequently, by~\eqref{E:may-9-five-bis} and~\eqref{101},
    \begin{align}\label{E:chwech}
        c \gtrsim
        \frac{|B|^{\frac{s-n}{n}}}
        {F^{-1}(|B|^{-1})\varphi(|B|^{\frac{1}{n}})}
        =
        \frac{\psi_{s,A}(|B|^{\frac{1}{n}})}{\varphi(|B|^{\frac{1}{n}})}\, .
    \end{align}
    Hence, owing to the arbitrariness of  $B$,
    \begin{align}\label{E:saith}
        \psi_{s,A}(r)
        \lesssim
        \varphi(r)
        \qquad\text{for $r>0$}.
    \end{align}
    The optimality of $\psi_{s,A}$ in~\eqref{100} and~\eqref{102} is thus established.
    \\ Part (ii). Assume first that $\varphi$ is an admissible function such that~ \eqref{limit1'} and \eqref{limit'} are satisfied. By \eqref{limit1'}, Part (i) ensures that  $V^{s,A}_{d,1}(\rn)\to \mathcal{L}^{\varphi}(\rn)$. Moreover, 
    owing to~\eqref{102},  
    \begin{align*}
 \lim_{r\to0^+}&\,
\sup_{u\in V^{s,A}_{d,1}(\rn), |\nabla^{[s]}u|_{\{s\}, A, \rn} \leq 1 } 
    \left( 
    \sup_{B\subset \rn, |B| \leq r} \; \frac 1{\varphi (|B|^{\frac{1}{n}})}\medint _B |u - u_B| \; dx
    \right)
    \\ 
& = 
 \lim_{r\to0^+}
    \sup_{B\subset \rn, |B| \leq r} \; \frac{\psi_{s,A}(|B|^{\frac{1}{n}})}
        {\varphi(|B|^{\frac{1}{n}})}  \left( 
        \sup_{|\nabla^{[s]}u|_{\{s\}, A, \rn} \leq 1 } 
        \frac{1}
        {\psi_{s,A}(|B|^{\frac{1}{n}})}\medint _B |u - u_B| \; dx
    \right)
    \\ 
    &\lesssim
        \lim_{r\to0^+}
        \sup_{B\subset\rn,\,|B|\le r}
        \frac{\psi_{s,A}(|B|^{\frac{1}{n}})}
        {\varphi(|B|^{\frac{1}{n}})}
        =0,
    \end{align*}
    and~\eqref{100v} follows owing to~\eqref{limit'}. \\ Conversely, suppose that $\varphi$ is an admissible function such that~\eqref{100v} holds. 
      Therefore, also~\eqref{100} holds, whence, by Part 1,   \eqref{limit1'} is satisfied. It remains to verify~\eqref{limit'}. Fix $\varepsilon>0$. By~\eqref{100v}, equation ~\eqref{E:vanishing-higher} is fulfilled. Hence, there exists $r>0$ such that 
    \begin{equation}\label{E:a-haon}
        \frac{1}{\varphi(|B|^{\frac{1}{n}})}
        \medint_B|u-u_B|\,dx<\varepsilon
    \end{equation}
    for every $u$ satisfying 
    $|\nabla^{[s]}u|_{\{s\},A;\rn}\le 1$ 
    and every ball 
    $B\subset\rn$ such that $|B|\le r$.   Fix  a ball $B$ centered at $0$ and satisfying $|B|\le r$. Let $v_f$ be defined as \eqref{E:uf}, with $f$ to be chosen later.  
 Equation ~\eqref{128} enables us to apply~\eqref{E:a-haon} to  $u=c^{-1}v_f$ for some appropriate positive constant $c$ independent of $B$ and $f$. Thereby, we obtain
    \begin{equation}\label{E:a-ceathair}
        \frac{1}{\varphi(|B|^{\frac{1}{n}})}
        \medint_B|v_f-(v_f)_B|\,dx<c\, \varepsilon,
    \end{equation}
    whence, via~\eqref{129}, 
    \begin{equation}\label{E:a-cuig}
        \frac{1}{\varphi(|B|^{\frac{1}{n}})}
\medint_B|v_f|\,dx<c\, \varepsilon.
    \end{equation}
Coupling the latter equation with \eqref{130} implies that
    \begin{equation}\label{E:a-se}
        \frac{|B|^{\frac{1}{n}}}{\varphi(|B|^{\frac{1}{n}})}
            \int_{|B|}^{\infty}
        f(r)\, r^{-1+\frac{s-1}{n}}\,dr
         \lesssim\varepsilon .
    \end{equation}    
    Owing to Fatou's lemma for Orlicz norms, an approximation argument shows that the inequality \eqref{E:a-se} continues to hold even the support of $f$ is not bounded.
 By  ~\eqref{E:dual-monotone}, there exists  a non-increasing function $g: (|B|, \infty) \to [0, \infty)$ such that  $\|g\|_{L^A(0,\infty)}\le1$ and
    \begin{equation}\label{E:a-do}
        \int_{|B|}^{\infty}
        g(r)\, r^{-1+\frac{s-1}{n}}\,dr
        \ge
        \frac{1}{2}
        \|r^{-1+\frac{s-1}{n}}
        \chi_{(|B|,\infty)}(r)\|_{L^{\widetilde{A}}(0,\infty)}.
    \end{equation}
    By   Lemma \ref{lemma-old},
    there exists a non-increasing function $f\colon [0,\infty) \to [0, \infty)$ such that $\|g\|_{L^A(0,\infty)}\gtrsim \|f\|_{L^A(0,\infty)}$, and
    \begin{equation}\label{E:a-tri}
        \int_{|B|}^{\infty}
        f(r)\, r^{-1+\frac{s-1}{n}}\,dr
        \gtrsim
        \|r^{-1+\frac{s-1}{n}}
        \chi_{(|B|,\infty)}(r)\|_{L^{\widetilde{A}}(0,\infty)}.
    \end{equation}
    Combining \eqref{E:a-se} with \eqref{E:a-tri} tells us that
    \begin{equation}\label{E:a-seacht}
        \frac{|B|^{\frac{1}{n}}}{\varphi(|B|^{\frac{1}{n}})}
        \|r^{-1+\frac{s-1}{n}}
        \chi_{(|B|,\infty)}(r)\|_{L^{\widetilde{A}}(0,\infty)}
    \lesssim\varepsilon.
    \end{equation}
    From ~\eqref{127} and~\eqref{E:may-9-five-bis}, one obtains that
    \begin{equation}\label{E:a-naoi}
        \frac{\psi_{s,A}(|B|^{\frac{1}{n}})}
        {\varphi(|B|^{\frac{1}{n}})}
        =
        \frac{|B|^{\frac{s-n}{n}}}{F^{-1}(|B|^{-1})\varphi(|B|^{\frac{1}{n}})}
        \lesssim\varepsilon.
    \end{equation}
   Hence, equation \eqref{limit'} follows since equation \eqref{E:a-naoi}  holds for every ball centered at zero and satisfying $|B|\le r$.   
\end{proof}

\section{Case $s>1$ and $k\geq 0$}\label{sec5}

This section concerns  the embeddings \eqref{emb-intro2}  and \eqref{emb-intro3} for   $s\in (1,\infty)\setminus\N$ 
and $k \geq 0$. Since the case $k=0$ was already dealt with in the preceding section,  in the proofs we shall assume that $k\geq 1$.

The following theorem about the admissible  choices of $s$ and $A$ in \eqref{emb-intro2} extends  Theorem \ref{teo3}.

\begin{theorem}[\bf Admissible $s$ and $A$ for $s>1$ and $k \geq 0$]\label{teo3bis}
Let $s\in (1, \infty)\setminus \N$ and let $A$ be a Young function.  Assume that the   embedding
\begin{equation}\label{july1k}
V_{d,k+1}^{s,A}(\rn) \to   \mathcal{L}^{k,\varphi} (\rn)
\end{equation}
holds for some  admissible function $\varphi$ and for some $k \in \{0, \dots, [s]-1\}$.  Then,
\begin{equation}\label{sk}
s\in (1, n+k+1)
\end{equation}
and
\begin{equation}\label{intzerok}
\int_0 \left (\frac t{A(t)}\right)^{\frac {s-(k+1)}{n-(s-(k+1))}} \; dt <\infty\,.
\end{equation}
\end{theorem}

Assume that $s\in (0, \infty)\setminus\N$,  $k\in \{0, \ldots, [s]-1\}$ and a Young function $A$ are such that the conditions \eqref{sk}  and \eqref{intzerok} are fulfilled.
Define  the Young function $F_k$  by
\begin{equation}\label{Fk}
{F_k}(t)= t^{\frac{n}{n-s+(k+1)}}
\int_0^t \frac{\widetilde{A}(\tau)}{\tau^{1+\frac{n}{n-s+(k+1)}}}\; d\tau \qquad \text{for $t \geq 0$,}
\end{equation}
and the function $\psi_{s,A}^k : (0, \infty) \to [0, \infty)$ as
\begin{equation}\label{varphi_k}
    \psi_{s,A}^k(r)=  
        \begin{cases}
      \displaystyle\frac{1}{r^{n-s+k} F_k^{-1}(r^{-n})} 
      & 
      \hbox{if}\;\; k\in\{0, \ldots, [s]-1\}
        \\
        \\
    r^{\{s\}} 
    A^{-1}\left(\frac{1}{r^n} 
    \right) & \hbox{if}\;\; k=[s]
    \end{cases}
\end{equation}
for $r>0$.
\\
Thanks to the property~\eqref{AAtilde}, the function $\psi_{s,A}^k$ admits the following equivalent expression:
\begin{equation}
    \label{E:alter-varphi_k}
    \psi_{s,A}^k(r)
    \approx  
    \begin{cases}
        \displaystyle
        {r^{k-s}}
        {\widetilde{F}_k^{-1}
        (r^{-n})} 
        & 
        \hbox{if}\;\; 
        k\in\{0, \ldots, [s]-1\}
            \\
            \\
        \displaystyle
        \frac{r^{\{s\}-n}}
        {\widetilde{A}^{-1}
         \left(\frac{1}{r^n} 
        \right)}
        & \hbox{if}\;\; k=[s]
    \end{cases}
\end{equation}
for $r>0$.
\\
Observe that $\psi_{s,A}^0=\psi_{s,A}$, where $\psi_{s,A}$ is defined by \eqref{101}.

The optimal Campanato target space in the embedding \eqref{emb-intro2} or \eqref{emb-intro3} is associated with the function $\psi_{s,A}^k$. This is asserted in Theorem \ref{teo_k}
below, which also provides us with a characterization of uniform embeddings into vanishing Campanato spaces.

Given $k \in \{0, \dots, [s]-1\}$, the embedding $V^{s,A}_{d,k+1}(\rn) \to V\!\mathcal{L}^{k, \varphi}(\rn)$ is said to hold \emph{uniformly} if $V^{s,A}_{d,k+1}(\rn) \to \mathcal{L}^{k, \varphi}(\rn)$ and
\begin{equation}\label{E:vanishing-higher-k}
    \lim_{r\to0^+}
    \sup_{u\in V^{s,A}_{d,k+1}, |\nabla^{[s]}u|_{\{s\}, A, \rn} \leq 1 } 
    \left( 
    \sup_{B\subset \rn, |B| \leq r} \; \frac 1{\varphi (|B|^{\frac{1}{n}})|B|^{\frac{k}{n
    }}}\medint _B |u - P^k_B[u]| \; dx
    \right)=0.
\end{equation}
An analogous definition can be given for the uniform embedding $V^{s,A}(\rn) \to V\!\mathcal{L}^{[s], \varphi}(\rn)$.

\begin{theorem}
[\bf Optimal embeddings into Campanato spaces for $s>1$ and $k\geq 0$]
\label{teo_k}
Assume that $k\in\N\cup\{0\}$ and $s>1$ are such that $k\le[s]$, and $s\in (1, n+k+1)\backslash \N$.
Let $A$ be a Young function.
\\
(i) Assume that $k\in \{0, \ldots, [s]-1\}$ and that $A$ fulfills the condition \eqref{intzerok}. Then, 
\begin{equation}\label{201}
 V^{s,A}_{d, k+1}(\rn) \to \mathcal{L}^{k, \psi_{s,A}^k}(\rn)  \end{equation}
and there exists a constant $c=c(n,s,k)$ such that
\begin{equation}\label{202}
   |u|_{\mathcal{L}^{k, \psi_{s,A}^k}(\rn)} \leq c \, \big| \nabla ^{[s]} u\big|_{\{s\}, A, \rn}
\end{equation}
for $u\in  V^{s,A}_{d, k+1}(\rn)$. Moreover, $\mathcal{L}^{k,\psi_{s,A}^k}(\rn)$ is the optimal  Campanato target space in \eqref{201} and \eqref{202}.
\\
If $\varphi$ is an admissible function, then
\begin{equation}
    \label{E:higher-vanishing-1}
    V^{s, A}_{d,k+1}(\rn) \to 
    V\!\mathcal{L}^{k, \varphi}(\rn)
    \qquad \text{uniformly}
    \end{equation}
if and only if 
\begin{align}
    \label{naples1}
\psi_{s,A}^k(r) \leq 
    \varphi (r) \qquad \text{for $r>0$} 
\end{align}
and
\begin{equation}
\label{E:higher-vanishing-2}
    \lim_{r \to 0^+} 
    \frac{\psi_{s,A}^k(r)}
    {\varphi (r)} =0.
\end{equation}
\\
(ii) Assume that $k= [s]$. Then,
\begin{equation}\label{203}
 V^{s,A}(\rn) \to \mathcal{L}^{[s], \psi_{s,A}^{[s]}}(\rn)  \end{equation}
and there exists a constant $c=c(n,s)$ such that
\begin{equation}\label{204}
|u|_{\mathcal{L}^{[s], \psi_{s,A}^{[s]}}(\rn)} \leq c \, \big| \nabla ^{[s]} u\big|_{\{s\}, A, \rn}
\end{equation}
for  $u\in  V^{s,A}(\rn)$. Moreover, $\mathcal{L} ^{[s], \psi_{s,A}^{[s]}}(\rn) $ is the optimal  Campanato target space in \eqref{203} and \eqref{204}.
\\
If $\varphi$ is an admissible function, then
\begin{equation}
    \label{E:higher-vanishing-3}
    V^{s, A}(\rn) \to 
V\!\mathcal{L} ^{[s], \varphi}(\rn)
    \qquad \text{uniformly}
    \end{equation}
if and only if
\begin{align}
    \label{naples2}
\psi_{s,A}^{[s]}(r) \leq 
    \varphi (r) \qquad \text{for $r>0$} 
\end{align}
and
\begin{equation}
\label{E:higher-vanishing-4}
    \lim_{r \to 0^+} 
    \frac{\psi_{s,A}^{[s]}(r)}
    {\varphi (r)} =0.
\end{equation}
\end{theorem}

We begin with the proof of Theorem \ref{teo3bis}.

\begin{proof}[Proof of Theorem \ref{teo3bis}] The outline of this proof is analogous to that of Theorem~\ref{teo3}. Assume that the embedding \eqref{july1k} holds for some $s\in (1, \infty) \setminus \N$.
Let $\xi \in C^\infty_0(\rn)$  be  a nonnegative function such that
\begin{align}\label{prague1}
\xi (x) = H(x) \qquad \text{in a neighborhood of $0$,}
\end{align}
where $H$ is a homogeneous harmonic polynomial of degree $k+1$.
 For each $j \in \N$,   consider the function $u_j \colon  \rn \to \mathbb  R$ defined as
\begin{equation}\label{100b}
u_j (x)= j^{s-n} \xi \Big(\frac x j\Big) \qquad \hbox{for}\;\; x\in \rn\,.
\end{equation}
The same argument which yields  the inequality \eqref{may30} in the proof of Theorem \ref{teo3}
shows that $u_j \in V^{s,A}_{d,k+1}$ and  there exists a constant $c$, independent of $j$, such that
\begin{equation}\label{may30bis}
|\nabla ^{[s]}u_j|_{\{s\}, A, \rn} \leq c.
\end{equation}
The embedding \eqref{july1k} and the inequality \eqref{may30bis} imply that there exists a constant $c$ such that
\begin{equation}\label{july2k}
|u_j|_{\mathcal L^{k,\varphi}(\rn)} \leq c
\end{equation}
for every $j \in \N$.
\\ Next, as above, denote by $B_r(0)$ the ball, centered at $0$, with radius $r$.
The coefficients of the polynomial $P^k_{B_{1}(0)}[u_j]$ are linear combinations of the components of
\begin{equation}\label{dec15}
\int_{B_{1}(0)}u_j (x)\, dx, \dots , \int_{B_{1}(0)}\nabla ^ku_j (x)\, dx.
\end{equation}
These integrals agree with
\begin{equation}\label{dec10}
j^n\int_{B_{1/j}(0)}\xi (x)\, dx, \dots , j^{n-k}\int_{B_{1/j}(0)}\nabla ^k\xi (x)\, dx.
\end{equation}
 By \eqref{prague1}
and \cite[Lemma 5.4]{CCPS_HC},
\begin{equation}\label{dec11}
P^k_{B_{1}(0)}[u_j]=0,
\end{equation}
provided that $j$ is sufficiently large. Therefore,
\begin{align}\label{dec13}
|u_j|_{\mathcal L^{k,\varphi}(\rn)} & = \sup_{B} \frac 1{\varphi (|B|^\frac 1n)|B|^\frac kn} \medint_{B}\big|u_j-P^k_B[u_j]\big|\, dx
\geq    \frac 1{\varphi \big(\omega_n^\frac 1n\big)} \medint_{B_1(0)}|u_j-P^k_{B_{1}(0)}[u_j]|\, dx
\\ \nonumber  & =    \frac 1{\varphi \big(\omega_n^\frac 1n\big)}\medint_{B_1(0)}|u_j|\, dx =  \frac {j^{s-n}}{\varphi \big(\omega_n^\frac 1n\big)} \medint_{B_1(0)}|\xi(x/j)|\, dx =  c\, \frac {j^{s-n}}{\varphi \big(\omega_n^\frac 1n\big)} \medint_{B_{1/j}(0)}|\xi(x)|\, dx
\\ \nonumber & = c \,j^{s-n} \medint_{B_{1/j}(0)}|H(x)|\, dx \geq c'\, j^{s-n-k-1}
\end{align}
for suitable constants $c$ and $c'$, independent of $j$.
Equation \eqref{sk} follows from \eqref{july2k} and \eqref{dec13}, by letting $j \to \infty$.

 As for the necessity of the condition \eqref{intzerok}, we refer to the proof of the sharpness of Theorem \ref{teo_k}. Indeed, 
equation \eqref{E:opt-7} of that proof 
tells that
\begin{equation}\label{dec20}
 \| r ^{-1+\frac{s-k-1}n}\|_{L^{\widetilde A}(1, \infty)}<\infty,
\end{equation}
and the finiteness of this norm is equivalent to \eqref{intzerok}.
\end{proof}

The proof of Theorem \ref{teo_k} makes use of the following lemmas.

\begin{lemma}\label{L:1-half}
    Let $n \geq 2$ and $k\in\N$. Then, there exists a constant $c=c(n,k)$   such that
    \begin{equation}\label{E:L1}
        \left\|u  - {P}^{k-1}_{B}[u] \right\|_{L^{1}(B)} \leq c\, |B|^{\frac{k}{n}}\left\|\nabla^{k}u\right\|_{L^1(B)}
    \end{equation}
    for every $u\in W^{k,1}_{\rm loc}(\rn)$ and every ball $B\subset\rn$.
\end{lemma}

\begin{proof} One has that
\begin{equation}\label{E:L1-1}
        \left\|u  - {P}^{k-1}_{B}[u] \right\|_{L^{\frac{n}{n-k}}(B)}
        \leq c\,\left\|\nabla^{k}u\right\|_{L^1(B)} \qquad \  \text{for} \ k\in\{1,\dots, n-1\} 
    \end{equation}
    and
    \begin{equation}\label{E:L1-2}
        \left\|u  - {P}^{k-1}_{B}[u] \right\|_{L^{\infty}(B)}
        \leq c \, |B|^{-1 +\frac{k}{n}}\left\|\nabla^{k}u\right\|_{L^1(B)} \qquad \ \text{for} \  k\geq n
    \end{equation}
    for some constant $c=c(n,k)$ and for every $u$ and $B$ as in the statement --
see e.g. \cite[Lemma~3.1]{CCPS_HC}.
\\ Assume first that $k\le n-1$. By H\"{o}lder's inequality and the inequality \eqref{E:L1-1},
    \begin{equation}\label{E:L:1-half-1}
        \int_{B} |u - P^{k-1}_{B}[u] | \, dx
        \leq
        |B|^{\frac{k}{n}}\left\|u-P^{k-1}_{B}[u]\right\|_{L^{\frac{n}{n-k}}(B)}\leq c\, |B|^{\frac{k}{n}}\int_{B}|\nabla ^{k}u|\, dx,
    \end{equation}
   namely \eqref{E:L1}. Next, suppose that $k\ge n$. By the inequality \eqref{E:L1-2},
    \begin{equation}\label{E:L:1-half-2}
        \int_{B} |u - P^{k-1}_{B}[u] | \, dx
        \leq
       |B|\, \left\|u-P^{k-1}_{B}[u]\right\|_{L^{\infty}(B)}\leq c \,|B|^{\frac{k}{n}}\int_{B}|\nabla ^{k}u|\, dx,
    \end{equation}
    whence \eqref{E:L1} follows.
\end{proof}

\begin{lemma}\label{lemma1}
Let $k\in\N\cup\{0\}$ and $s>1$ be such that $k\le[s]$ and $s\in (1, n+k+1)\backslash \N$.
Let $H$ be a homogeneous harmonic polynomial of degree $k+1$.
Assume that $f$ is a nonnegative, non-increasing function in $L^A(0, \infty)$ with bounded support.
Let $w_f : \rn \to \R$ be the function defined by
\begin{equation}\label{226}
w_f(x) = H(x)\, \int_{\omega_n \,|x|^n}^\infty \int_{r_1}^\infty \dots \int_{r_{[s]}}^\infty  f(r_{[s] +1})\, r_{[s] +1}^{-[s] -1+\frac{s-(k+1)}{n}}
\;dr_{[s]+1} \dots dr_1 \qquad \text{for $x\in \rn$.}
\end{equation}
Then,
\begin{equation}\label{sep109}
|\nabla^{[s]}  w_f |_{\{s\}, A, \rn} \leq c \,\|f\|_{L^A(0, \infty)}\,
\end{equation}
for some constant $c=c(n,k,s)$.
\end{lemma}

\begin{proof}  Throughout this proof the constants involved in the inequalities and in the relations $\approx$ and $\lesssim$ depend only on  $n,k,s$.
To begin with, notice the alternative expression for $w_f$:
$$
w_f(x) = \frac {1}{[s]!} \, H(x) \,\int_{\omega_n |x|^n}^\infty f(r)\,  r^{-[s]-1+ \frac{s-(k+1)}{n}} \, (r - \omega_n |x|^n)^{[s]} \; dr \qquad
\hbox{for}\;\; x\in \rn$$
which follows from \eqref{226} via Fubini's theorem and will be used without explicit mention.
\\ To prove the inequality  \eqref{sep109}, it suffices to show that
\begin{equation}\label{150}
\int_{\rn} \int_{\rn} A\left( \frac{ | \nabla ^{[s]} w_f(x) - \nabla^{[s]} w_f(y)|}{|x-y|^{\{s\}}}\right) \; \frac{dx \, dy}{|x-y|^n} \leq \int_0^\infty A\left ( c \, f(r)\right)\; dr
\end{equation}
for some constant  $c$. 
One can verify that any ${[s]}-$th order partial derivative of $w_f$ is a linear combination of finitely many terms of the form
\begin{equation}\label{dec240}
x_{\alpha_1} \cdots x_{\alpha_i} \, |x|^{h} \, \int_{\omega_n |x|^n}^\infty f(r)\, r_{}^{ -{[s]} -1 + \frac {s-(k+1)}n}(r-\omega_n |x|^n)^m\; d r,
\end{equation}
where $i$ and $m$ are nonnegative integers, $m \leq [s]$, $\alpha_i \in\{1, \dots , n\}$ and $h$ is an integer such that
\begin{equation}\label{E:a}
 h =  n([s] -m)-i-[s] +k+1.
\end{equation} 
In addition, we notice that the term with $m=[s]$ appears only when $k=[s]$, or $k+1=[s]$.
\\
To estimate the left-hand side of the inequality \eqref{150},  we observe that, by symmetry, one may restrict the region of integration to the set $\{(x,y)\in \rn\times \rn: |x|\leq |y|\}$, and we split the latter into
 the subsets
 $\{|y| \geq 2 |x|\}$ and $\{2|x| > |y| \geq |x|\}$.
\\
Part 1: Estimate over  $\{|y| \geq 2 |x|\}$.
Let us first assume that  ${[s]}\neq k+1$. Trivially,
\begin{equation*}
|\nabla^{[s]} w_f(x) - \nabla^{[s]} w_f(y)| \leq |\nabla^{[s]} w_f(x)| + |\nabla^{[s]} w_f(y)| \qquad \text{for a.e. $x, y \in \rn$.}
\end{equation*}
We shall provide a bound for $|\nabla^{[s]} w_f(x)|$,  the bound for $|\nabla^{[s]} w_f(y)|$ being analogous, with $x$ replaced with $y$. 
 Thanks to the monotonicity of $f$, we deduce from \eqref{dec240} that $|\nabla^{[s]} w_f(x)|$ is bounded by a linear combination of finitely many terms of the form
 $$|x|^{h+i} f(\omega_n |x|^n) \int_{\omega_n |x|^n}^\infty r^{-[s] -1+\frac{s-(k+1)}{n}+m}\,dr,$$
 which satisfy the chain:
\begin{align}\label{E:estimate-y-2x}
|x|^{h+i} f(\omega_n |x|^n) \int_{\omega_n |x|^n}^\infty r^{-[s] -1+\frac{s-(k+1)}{n}+m}\,dr
&  \approx |x|^{h+i -n[s]+s-(k+1)+n\, m} f(\omega_n |x|^n)
\\ & \nonumber =|x|^{s-[s]} f(\omega_n |x|^n)= |x|^{\{s\}} f(\omega_n |x|^n).
\end{align}
The convergence of the integral in \eqref{E:estimate-y-2x} is due to the fact that
\begin{equation}\label{E:need-to-show}
[s]>\frac{s-(k+1)}{n}+m.
\end{equation}
To verify this inequality, observe that, owing to our assumptions,  $\frac{s-(k+1)}{n}<1$. Hence, the inequality~\eqref{E:need-to-show} is satisfied whenever $m \leq [s] -1$. In addition, since  $[s] \neq k+1$, one can have $m = [s]$ only  when $k=[s]$. In this case, the inequality~\eqref{E:need-to-show} reduces to the trivial inequality $s- [s] <1$. Let us also  note that the last but one equality in~\eqref{E:estimate-y-2x} holds thanks to~\eqref{E:a}.
 \\ Thanks to the above considerations, one has that
\begin{align*}
\int_{\Rn} &\int_{\{|y|\geq 2|x|\}} A\left(\frac{|\nabla^{[s]} w_f(x) - \nabla^{[s]} w_f(y)|}{|x-y|^{\{s\}}}\right) \frac{\,dx \,dy}{|x-y|^n}\\
&\lesssim
\int_{\Rn} \int_{\{|y| \geq 2|x|\}} A \left( c |x|^{\{s\}} |y|^{-\{s\}} f(\omega_n|x|^n) \right) \frac{dy}{|y|^n}\,dx
+ \int_{\Rn} \int_{\{|y| \geq 2|x|\}} A \left(c f(\omega_n|y|^n) \right) \,dx\frac{dy}{|y|^n}\\
&\lesssim \int_{\Rn} \int_{2|x|}^\infty A \left( c' |x|^{\{s\}} r^{-\{s\}} f(\omega_n|x|^n) \right) \frac{dr}{r}\,dx
+ \int_{\Rn} A \left( c' f(\omega_n|y|^n) \right) \,dy
\end{align*}
for some constants $c$ and $c'$.
The change of variables $t=c' |x|^{\{s\}} r^{-\{s\}} f(\omega_n|x|^n)$ yields
$$
\int_{2|x|}^\infty A \left( c' |x|^{\{s\}} r^{-\{s\}}  f(\omega_n|x|^n) \right) \frac{dr}{r}
\approx \int_0^{c' f(\omega_n |x|^n)} A(t) \frac{dt}{t}
\leq A\left(c' f(\omega_n |x|^n)\right),
$$
where the inequality holds by the property \eqref{monotone}.
Therefore,
\begin{align}\label{nov209}
\int_{\Rn} \int_{\{|y|\geq 2|x|\}} A\left(\frac{|\nabla^{[s]} w_f(x) - \nabla^{[s]} w_f(y)|}{|x-y|^{\{s\}}}\right) \frac{\,dx \,dy}{|x-y|^n}
\lesssim \int_{\Rn} A \left(c f(\omega_n|x|^n) \right) \,dx
= \int_0^\infty A \left(c f(r) \right) \,dr
\end{align}
for some constant $c$.
\\
 Assume next that ${[s]}=k+1$. Besides the terms in \eqref{dec240} with $m \leq [s]-1$, which can be estimated as above via an  expression of the form
 $$ \int_0^\infty A \left(c f(r) \right) \,dr,$$
 we have now also to  estimate the term with $m={[s]}$. Note that, in this case, $h=i=0$, and the expression \eqref{dec240} can be rewritten as
$$\int_{\omega_n|x|^n}^{\infty} \int_{r_1}^\infty \dots \int_{r_{[s]}}^\infty f(r_{ [s]+1} )\, r_{[s]+1} ^{  -[s]-1+\frac{s-(k+1)}{n}}\,dr_{[s]+1} \dots dr_1.$$
Altogether, 
 we obtain that
\begin{align*}
\int_{\R^n}\int_{\{|y|\ge2|x|\}} &
	A\left(\frac{|\nabla^{[s]} w_f(x)-\nabla^{[s]} w_f(y)|}{|x-y|^{\{s\}}}\right)\frac{dx\,dy}{|x-y|^n}\\ &  \lesssim
\int_{\R^n}\int_{\{|y|\ge2|x|\}}
			A\left(\frac{c\int_{\omega_n|x|^n}^{\omega_n|y|^n} \int_{r_1}^\infty \dots \int_{r_{[s]}}^\infty f(r_{ [s]+1} )\, r_{[s]+1} ^{-[s]-1+\frac{s-(k+1)}{n}}\,dr_{[s]+1} \dots dr_1}{|y|^{\{s\}}}\right)\frac{dy}{|y|^n}\,dx\\
   & \quad + \int_0^\infty A \left(c f(r) \right) \,dr
\end{align*}
for some constant $c$.
The following chain holds:
\begin{align}\label{nov200}
		&\int_{\R^n}\int_{\{|y|\ge2|x|\}}
			A\left(\frac{c\int_{\omega_n|x|^n}^{\omega_n|y|^n} \int_{r_1}^\infty \dots \int_{r_{[s]}}^\infty f(r_{[s]+1} )\, r_{[s]+1} ^{-[s]-1+\frac{s-(k+1)}{n}}\,dr_{[s]+1} \dots dr_1}{|y|^{\{s\}}}\right)\frac{dy}{|y|^n}\,dx
				\\  \nonumber
		& \lesssim \int_{\R^n}\int_{\{|y|\ge2|x|\}}
			\int_{\omega_n|x|^n}^{\omega_n|y|^n} \int_{r_1}^\infty \dots \int_{r_{[s]}}^\infty A\left(c'f(r_{[s]+1})\right)r_{[s]+1}^{-[s]-1+\frac{s-(k+1)}{n}}\,dr_{[s]+1} \dots dr_1\frac{dy}{|y|^{n+\{s\}}}\,dx
\\  \nonumber
		& \lesssim \int_{\R^n}\int_{\{|y|\ge2|x|\}}
			\int_{\omega_n|x|^n}^{\omega_n|y|^n}  A\left(c'f(r)\right)r^{-1+\frac{s-(k+1)}{n}}\,dr\frac{dy}{|y|^{n+\{s\}}}\,dx
				\\  \nonumber
		&
\leq \int_{0}^{\infty}A\left(c'f(r)\right)r^{-1+\frac{s-(k+1)}{n}}
				\int_{\{\omega_n|x|^n<r\}}\int_{\{|y|\ge2|x|\}}\frac{dy}{|y|^{n+\{s\}}}\,dx\,dr
				\\  \nonumber
		& \lesssim \int_{0}^{\infty}A\left(c'f(r)\right)r^{-1+\frac{s-(k+1)}{n}}
				\int_{\{\omega_n|x|^n<r\}}|x|^{-\{s\}}\,dx\,dr
				\\  \nonumber
		& \lesssim \int_{0}^{\infty}A\left(c'f(r)\right)r^{-1+\frac{\{s\}}{n}}r^{1-\frac{\{s\}}{n}}
				\,dr
				 = \int_{0}^{\infty}A\left(c'f(r)\right)\,dr
\end{align}
for some positive constant  $c'$.
Observe that the first inequality  relies upon Jensen's inequality. Hence, the inequality \eqref{nov209} holds  for ${[s]}=k+1$ as well.

\noindent
Part 2: Estimate over $\{2|x| \geq |y| \geq |x|\}$. As above, we first assume that ${[s]} \neq k+1$.
Set $j=[s] -m \geq 0$. Then, $|\nabla^{[s]} w_f(x) - \nabla^{[s]} w_f(y)|$ is bounded by a finite linear combination of terms of the form
\begin{align*}
& \bigg |x_{\alpha_1} \cdots  x_{\alpha_i} \, |x|^{h} \, \int_{\omega_n |x|^n}^\infty \int_{r_{j+1}}^\infty \cdots
\int_{r_{[s]}}^\infty f(r_{{[s]}+1})\, r_{{[s]}+1}^{ -{[s]} -1 + \frac {s-(k+1)}n}\; d r_{{[s]}+1} \cdots  d r_{j+1} \nonumber
\cr& \quad \quad \quad \quad \quad \quad\quad \quad \quad
- y_{\alpha_1} \cdots y_{\alpha_i} \, |y|^{h} \, \int_{\omega_n |y|^n}^\infty \int_{r_{j+1}}^\infty \cdots
\int_{r_{[s]}}^\infty f(r_{{[s]}+1})\, r_{{[s]}+1}^{ -{[s]} -1 + \frac {s-(k+1)}n}\; d r_{{[s]}+1} \cdots  d r_{j+1}\bigg |.  
\end{align*}
The latter expression is bounded by
\begin{align*}
 &
 \bigg |x_{\alpha_1} \cdots  x_{\alpha_i} \, |x|^{h} - y_{\alpha_1} \cdots y_{\alpha_i} \, |y|^{h}\bigg| \nonumber
\cr &  \quad \quad \quad \quad \quad \quad\quad \quad \quad
\times \,
\int_{\omega_n |y|^n}^\infty \int_{r_{j+1}}^\infty \cdots
\int_{r_{[s]}}^\infty f(r_{{[s]}+1})\, r_{{[s]}+1}^{ -{[s]} -1 + \frac {s-(k+1)}n}\; d r_{{[s]}+1} \cdots  d r_{j+1}\nonumber
\cr &
+  \bigg |x_{\alpha_1} \cdots  x_{\alpha_i} \, |x|^{h}\, \int_{\omega_n |x|^n}^{\omega_n |y|^n} \int_{r_{j+1}}^\infty \cdots
\int_{r_{[s]}}^\infty f(r_{{[s]}+1})\, r_{{[s]}+1}^{ -{[s]} -1 + \frac {s-(k+1)}n}\; d r_{{[s]}+1} \cdots  d r_{j+1}\bigg| \nonumber
\cr &
 = I + II
\end{align*}
for a.e. $x, y$ such that $2|x| \geq |y| \geq |x|$.
By \cite[Lemma 7.8]{ACPS_emb} (which still holds if the assumption $i\leq n$ appearing  in the statement is dropped), we deduce that, for the same $x$ and $y$,
\begin{align*}
I &\lesssim  |x-y|\, |x|^{h+i-1}\; \int_{\omega_n |y|^n}^\infty \int_{r_{j+1}}^\infty \cdots
\int_{r_{[s]}}^\infty f(r_{{[s]}+1})\, r_{{[s]}+1}^{ -{[s]} -1 + \frac {s-(k+1)}n}\; d r_{{[s]}+1} \cdots  d r_{j+1}
\cr &
\lesssim  |x-y|\, |x|^{h+i-1}\; \int_{\omega_n |y|^n}^\infty f(r)\, r^{ -j-1 + \frac {s-(k+1)}n}\; d r
\leq |x-y|\, |x|^{h+i-1}\; \int_{\omega_n |x|^n}^\infty f(r)\, r^{ -j -1 + \frac {s-(k+1)}n}\; d r \nonumber
\cr &
\lesssim  |x-y|\, |x|^{h+i-1}\;  f(\omega_n |x|^n)\, |x|^{ -jn  + s-(k+1)}
=  |x|^{\{s\} -1} |x-y|\,  f(\omega_n |x|^n), \nonumber
\end{align*}
where the third inequality holds since $|y|\geq |x|$, and the last one is due to~\eqref{E:a} and \eqref{E:need-to-show}.
\\
On the other hand,
\begin{align*}
II &\lesssim  |x|^{h+i} f(\omega_n |x|^n) \int_{\omega_n |x|^n}^{\omega_n |y|^n} \int_{r_{j+1}}^\infty \cdots
\int_{r_{[s]}}^\infty  r_{{[s]}+1}^{ -{[s]} -1 + \frac {s-(k+1)}n}\; d r_{{[s]}+1} \cdots  d r_{j+1}\\
\nonumber &
\lesssim  |x|^{h+i} f(\omega_n |x|^n) \int_{\omega_n |x|^n}^{\omega_n |y|^n} r^{-j-1+\frac{s-(k+1)}{n}}\,dr
\lesssim |x|^{h+i-jn+s-(k+1)-n} (|y|^n-|x|^n) f(\omega_n |x|^n)\\
\nonumber &
\lesssim  |x|^{\{s\} -1}\, |y-x|\,f(\omega_n |x|^n)
\end{align*}
for a.e. $x, y$ such that $2|x| \geq |y| \geq |x|$.
\\ If ${[s]}=k+1$, then the terms in the expression $\nabla w_f(x) - \nabla w_f(y)$ corresponding to $j \geq 1$ can be estimated as above, whereas
 the term   corresponding to  $j=i=h=0$ in \eqref{dec240} can be estimated, in absolute value, by
\begin{align*}
&\int_{\omega_n |x|^n}^{\omega_n |y|^n} \int_{r_{1}}^\infty \cdots
\int_{r_{[s]}}^\infty  f(r_{{[s]}+1}) \,r_{{[s]}+1}^{ -{[s]} -1 + \frac {s-(k+1)}n}\; d r_{{[s]}+1} \cdots  d r_{1}\\
&\lesssim  f(\omega_n |x|^n) \int_{\omega_n |x|^n}^{\omega_n |y|^n} r^{-1+\frac{s-(k+1)}{n}}\,dr
\lesssim |x|^{s-(k+1)-n} (|y|^n-|x|^n) f(\omega_n |x|^n)
\lesssim |x|^{\{s\}-1} |y-x| f(\omega_n |x|^n).
\end{align*}
Therefore,
\begin{align*}
\int_{\Rn} &\int_{\{|x| \leq |y| \leq 2|x|\}} A\left(\frac{|\nabla^{[s]} w_f(x) - \nabla^{[s]} w_f(y)|}{|x-y|^{\{s\}}}\right) \frac{\,dx \,dy}{|x-y|^n}\\
&\lesssim  \int_{\Rn} \int_{\{|x| \leq |y| \leq 2|x|\}} A\left(c|x|^{\{s\}-1}|y-x|^{ 1-\{s\}}f(\omega_n |x|^n)\right)\frac{dy}{|y-x|^n}\,dx\\
&= \int_{\Rn} \int_{\{|x| \leq |z+x| \leq 2|x|\}} A\left(c|x|^{\{s\}-1}|z|^{1-\{s\}}f(\omega_n |x|^n)\right)\frac{dz}{|z|^n}\,dx\\
&\leq\int_{\Rn} \int_{\{|z|\leq 3|x|\}} A\left(c|x|^{\{s\}-1}|z|^{1-\{s\}}f(\omega_n |x|^n)\right)\frac{dz}{|z|^n}\,dx\\
&\lesssim \int_{\Rn} \int_{0}^{3|x|} A\left(c|x|^{\{s\}-1} r^{1-\{s\}}f(\omega_n |x|^n)\right)\frac{dr}{r}\,dx
\end{align*}
for some positive constant $c$.
The change of variables $t=c|x|^{\{s\}-1} r^{1-\{s\}}f(\omega_n |x|^n)$ yields
$$
\int_{0}^{3|x|} A\left(c|x|^{\{s\}-1}r^{ 1-\{s\}}f(\omega_n |x|^n)\right)\frac{dr}{r}
\approx \int_0^{c f(\omega_n |x|^n)} A(t) \frac{\,dt}{t}
\leq A(c \,f(\omega_n |x|^n))
$$
for $x\in \rn$.
Thereby,
\begin{align}\label{nov210}
\int_{\Rn} \int_{\{|x|\leq|y|\leq 2|x|\}} A\left(\frac{|\nabla^{[s]} w_f(x) - \nabla^{[s]} w_f(y)|}{ |x-y|^{\{s\}}}\right) \frac{\,dx \,dy}{|x-y|^n}
\lesssim \int_{\Rn} A \left( c f(\omega_n|x|^n) \right) \,dx
= \int_0^\infty A \left(c f(r) \right) \,dr.
\end{align}
The inequality \eqref{150} follows from \eqref{nov209} and \eqref{nov210}.
\end{proof}

We are now in a position to accomplish the proof of Theorem \ref{teo_k}.

\begin{proof}[Proof of Theorem \ref{teo_k}] As observed above, we may focus on the case when $k \geq 1$, the case when $k=0$ having already been dispensed with in Section \ref{sec4}. Throughout this proof, the constants in the inequalities and in the relations $\lq\lq \lesssim"$ and $\lq\lq \approx "$   depend only on $n,s$ and $k$.
    \\\ Part (i).
Owing to Theorem~\ref{thm_Z} applied with $s$ replaced with $s-k-1$, an application which is permitted by~\eqref{intzerok} and the assumption $s<n+k+1$,
we have that
    \begin{equation*}
        \|v\|_{X_{s-k-1}(\mathbb R^n)}\lesssim |\nabla^{[s]-(k+1)}v|_{\{s\},A,\mathbb R^n}
    \end{equation*}
    for $v\in V^{s-(k+1),A}_{d,0}(\mathbb R^n)$. 
    If $u\in V^{s,A}_{d,k+1}(\rn)$, then each partial derivative of $u$ of order $k+1$ belongs to $V_{d,0}^{s-(k+1),A}(\rn)$, and 
    \begin{equation}\label{E:T4.4-1}
        \|\nabla^{k+1}u\|_{X_{s-k-1}(\mathbb R^n)}\lesssim |\nabla^{[s]}u|_{\{s\},A,\mathbb R^n}.
    \end{equation}
    Fix any ball $B\subset\rn$ and $u\in V^{s,A}_{d,k+1}(\mathbb R^n)$. 
    By Lemma~\ref{L:1-half}, applied with $k$ replaced with $k+1$, we have
   \begin{equation}\label{E:T:4.4.-4}
        \medint_{B} |u - P^k_{B}[u]| \; dx
         \lesssim\,
        |B|^{\frac{k+1}{n}-1}
        \int_{B}|\nabla ^{k+1}u|\; dx.
    \end{equation}
    Thanks to ~\eqref{E:T:4.4.-4}, \eqref{holder}, \eqref{E:T4.4-1} and~\eqref{emb-ass}, and to  the fact that $\chi_B^{**}=\chi_{(0,|B|)}^{**}$,
    one deduces that
\begin{align}\label{205}
\medint_{B} | u - P^k_{B}[u] | \; dx
&
\lesssim |B|^{\frac{k+1}{n}-1} \|\nabla^{k+1} u\|_{X_{s-k-1}(\rn)}\, \| \chi_{B}\|_{X_{s-k-1}'(\rn)}
\\ &
\lesssim\, |B|^{\frac{k+1}{n}-1} |\nabla^{[s]} u|_{\{s\}, A, \rn} \,
\big\|r^{\frac{s-(k+1)}{n}} \left (\chi_{(0, |B|)}\right)^{**}(r)\big\|_{L^{\widetilde{A}}(0, \infty)}. \nonumber
\end{align}
\\
An application of Lemma~\ref{L:lemma-for-ri-norms}, with $r=|B|$, $\alpha =\frac{s-(k+1)}{n}$, and $\beta=0$
yields
\begin{equation}\label{E:norm-higher}
    \big\|r^{\frac{s-(k+1)}{n}} \left (\chi_{(0,|B|)}\right)^{**}(r)\big\|_{L^{\widetilde{A}}(0, \infty)}
    \approx
    |B|\big\|r^{\frac{s-(k+1)}{n}-1} \chi_{(|B|, \infty)}(r) \big\|_{L^{\widetilde{A}}(0, \infty)}.
\end{equation}
On the other hand, computations show that
\begin{equation}\label{E:opt-0}
    \big\|r^{\frac{s-(k+1)}{n}-1} \chi_{(|B|, \infty)}(r) \big\|_{L^{\widetilde{A}}(0, \infty)}
    \approx
    \frac{1}{|B|^{1-\frac{s-(k+1)}{n}}F_k^{-1}\left(|B|^{-1}\right)},
\end{equation}
see e.g. \cite[Equation (6.24)]{CiRa}.
Coupling \eqref{E:norm-higher} with \eqref{E:opt-0} results in
\begin{equation}\label{210}
    \big\|r^{\frac{s-(k+1)}{n}} \left (\chi_{(0,|B|)}\right)^{**}(r)\big\|_{L^{\widetilde{A}}(0, \infty)}
    \approx
   \frac{|B|^{\frac{s-(k+1)}{n}}}{F_k^{-1}\left(|B|^{-1}\right)}.
\end{equation}
Equations \eqref{205}, \eqref{210} and~ \eqref{varphi_k} imply that
\begin{align}\label{212}
&\frac{1}{\psi_{s,A}^k(|B|^{\frac{1}{n}}) \,|B|^{\frac{k}{n}}}\, \medint_{B} |u - P^k_{B}[u]| \; dx
\lesssim \frac{|B|^{\frac{k+1}{n}-1+\frac{s-(k+1)}{n}}}{\psi_{s,A}^k(|B|^{\frac{1}{n}}) \,|B|^{\frac{k}{n}} {F}^{-1}_k (|B|^{-1})}
|\nabla ^{[s]}u|_{\{s\}, A, \rn}
\\ &
 =  \frac{|B|^{\frac{s-k}{n}-1}}{\psi_{s,A}^k(|B|^{\frac{1}{n}}) {F}^{-1}_k (|B|^{-1})}
|\nabla ^{[s]}u|_{\{s\}, A, \rn}
 =   |\nabla ^{[s]}u|_{\{s\}, A, \rn}. \nonumber
\end{align}
Taking the supremum over all balls in \eqref{212} and using~\eqref{min} yield~\eqref{202}.
\\
To prove the optimality of the space $\mathcal L^{k,\psi_{s,A}^{ k}}(\rn)$, assume that $\varphi$ is an admissible function such that
\begin{equation}\label{202bis}
   |u|_{\mathcal{L}^{k,\varphi}(\rn)} \lesssim \, \big| \nabla ^{[s]} u\big|_{\{s\}, A, \rn}
\end{equation}
for $u\in  V^{s,A}_{d, k+1}(\rn)$.
Let $H$ be a homogeneous harmonic polynomial of degree $k+1$ and let $f$ be a nonnegative, non-decreasing function in $L^A(0,\infty)$ with bounded support. Let $w_f$ be defined by~\eqref{226}.
Then, by~\eqref{sep109},  
\begin{equation}\label{E:opt-1}
   |\nabla^{[s]}w_f|_{\{s\},A,\rn} \lesssim \|f\|_{L^A(0,\infty)}.
\end{equation}
Let $B$ be a ball centered at $0$.
Analogous arguments as in the proofs of  \cite[Equations (6.5) and (6.7)]{CCPS_HC}
tell us that
\begin{equation}\label{E:opt-3}
    P^k_B[w_f]= 0
\end{equation}
and
\begin{equation}\label{E:opt-4}
    \medint_B|w_f|\,dx \gtrsim |B|^{\frac{k+1}{n}}\int_{|B|}^{\infty}f(r)\, r^{-1+\frac{s-(k+1)}{n}}\,dr.
\end{equation}
From  \eqref{202bis}--\eqref{E:opt-4}, one infers that 
\begin{align}\label{E:opt-5}
  \int_{|B|}^{\infty}f(r)\,r^{-1+\frac{s-(k+1)}{n}}\,dr&\le |B|^{-\frac{k+1}{n}}\medint_B|w_f|\,dx
        =|B|^{-\frac{k+1}{n}}\medint_B|w_f-P^k_B[w_f]|\,dx
            \\
        &\le \varphi (|B|^{\frac{1}{n}})|B|^{-\frac{1}{n}}|w_f|_{\mathcal L^{k,\varphi}(\rn)}
        \lesssim \varphi(|B|^{\frac{1}{n}})|B|^{-\frac{1}{n}}\|f\|_{L^A(0,\infty)},\nonumber
\end{align}
whence
\begin{equation}\label{E:opt-6}
    \frac{1}{\|f\|_{L^A(0,\infty)}} \int_{|B|}^{\infty}f(r)\,r^{-1+\frac{s-(k+1)}{n}}\,dr
    \lesssim \varphi (|B|^{\frac{1}{n}})|B|^{-\frac{1}{n}}.
\end{equation}
Notice that, thanks to an approximation argument and the Fatou property of Orlicz spaces, the inequality \eqref{E:opt-6} continues to hold even if $f$  is not necessarily with bounded support.
Thus,
\begin{equation}\label{E:opt-6a}
    \sup_{f \in L^A(0,\infty), \,f \downarrow}\frac{1}{\|f\|_{L^A(0,\infty)}} \int_{|B|}^{\infty}f(r)\,r^{-1+\frac{s-(k+1)}{n}}\,dr
    \lesssim \varphi (|B|^{\frac{1}{n}})|B|^{-\frac{1}{n}}.
\end{equation}
By~\cite[Lemma~5.2]{ACPS_modulus}, this  supremum  does not  increase, up to a multiplicative constant, if the class of admissible functions $f$ is modified as to consist of 
    those $f\in L^A(0, \infty)$ which vanish on $(0,|B|)$ and do not increase on $(|B|,\infty)$. 
Therefore, by equations \eqref{E:opt-6a} and \eqref{E:dual-monotone},
\begin{equation}\label{E:opt-7}
    \|r^{-1+\frac{s-(k+1)}{n}}\chi_{(|B|,\infty)}(r)\|_{L^{\widetilde A}(0,\infty)}
    \lesssim \varphi (|B|^{\frac{1}{n}})|B|^{-\frac{1}{n}}.
\end{equation}
Combining~\eqref{E:opt-7} with \eqref{E:opt-0} yields:
\begin{equation}\label{E:opt-8}
    \frac{1}{|B|^{1-\frac{s-(k+1)}{n}}F_k^{-1}(|B|^{-1})} \lesssim \varphi(|B|^{\frac{1}{n}})|B|^{-\frac{1}{n}},
\end{equation}
whence, owing to~\eqref{varphi_k},
\begin{equation}\label{E:opt-8bis}
     \psi_{s,A}^k(|B|^{\frac{1}{n}}) \lesssim \varphi(|B|^{\frac{1}{n}}).
\end{equation}
Consequently, thanks to the arbitrariness of $B$, we obtain
\begin{equation}\label{E:opt-9}
    \psi_{s,A}^k(r)\lesssim\varphi(r) \qquad \text{for $r>0$.}
\end{equation}
The optimality of the space $\mathcal L^{k,\psi_{s,A}^k}$ in~\eqref{201} is thus established.

The equivalence of~\eqref{E:higher-vanishing-1}--\eqref{naples1} and~\eqref{E:higher-vanishing-2} can be proved completely analogously to the corresponding part of the proof of Theorem~\ref{teo2}, and is therefore omitted. 
\par \noindent
 Part (ii). Fix any function $u\in V^{s,A}(\rn)$ and an open ball $B\subset\rn$. Set
\begin{equation}\label{E:ii-1}
    v=u-P^{[s]}_B[u].
\end{equation}
Equation ~\eqref{Poli_cond} implies that
\begin{equation}\label{E:ii-2}
\int_{B}\nabla^hv\,dx=0
\end{equation}
for $h\in\{0,\dots,[s]\}$.
This ensures  that
\begin{equation}\label{E:ii-2.5}
    P^{[s]-1}_B[v] = 0
\end{equation}
and
\begin{equation}\label{E:ii-3}
    \int_{B}\nabla^{[s]}u\,dx = \int_{B}\nabla^{[s]}\big(P^{[s]}_B[u]\big)\,dx.
\end{equation}
Since $\nabla^{[s]}\big(P^{[s]}_B[u]\big)$ is  constant, 
\begin{equation}\label{E:ii-4}
\int_{B}\nabla^{[s]}\big(P^{[s]}_B[u]\big)\,dx = |B|\,\nabla^{[s]}\big(P^{[s]}_B[u]\big ).
\end{equation}
From ~\eqref{E:ii-3} and \eqref{E:ii-4}, we deduce that
\begin{equation}\label{E:ii-5}
    \big(\nabla^{[s]}u\big)_B = \nabla^{[s]}\big(P^{[s]}_B[u]\big).
\end{equation}
Equations  \eqref{E:ii-5} and~\eqref{E:ii-1} tell us that
\begin{equation}\label{E:ii-6}
\int_{B}\big|\nabla^{[s]}v\big|\,dx = \int_{B}\left|\nabla^{[s]}u-\nabla^{[s]}\big(P^{[s]}_B[u]\big)\right|\,dx = \int_{B}\big|\nabla^{[s]}u-\big(\nabla^{[s]}u\big)_B\big|\,dx.
\end{equation}
Lemma~\ref{L:1-half} implies that
\begin{equation}\label{E:ii-7}
    \int_{B}\big|v-P^{[s]-1}_B[v]\big|\,dx \lesssim |B|^{\frac{[s]}{n}} \int_{B}\big|\nabla^{[s]}v\big|\,dx.
\end{equation}
Combining equations ~\eqref{E:ii-1}, \eqref{E:ii-2.5}, \eqref{E:ii-7} and \eqref{E:ii-6} yields:
\begin{align}\label{E:ii-8}
        \int_{B}\big|u-P^{[s]}_B [u]\big |\,dx &= \int_{B}\big|v\big|\,dx =\int_{B}\big|v-P^{[s]-1}_B[v]\big|\,dx
            \\
        &\lesssim  |B|^{\frac{[s]}{n}} \int_{B}\big|\nabla^{[s]}v\big|\,dx
        =  |B|^{\frac{[s]}{n}} \int_{B}\big|\nabla^{[s]}u-\big(\nabla^{[s]}u\big)_B\big|\,dx.\nonumber
\end{align}
Hence,  
\begin{equation}\label{E:ii-9}
    \frac{1}{\psi_{s,A}^{[s]}\big(|B|^{\frac{1}{n}}\big)|B|^{\frac{[s]}{n}}}
    \medint_{B}\big|u-P^{[s]}_B[u]\big|\,dx
    \lesssim
    \frac{1}{\psi_{s,A}^{[s]}\big(|B|^{\frac{1}{n}}\big)}
    \medint_{B}\big|\nabla^{[s]}u-\big(\nabla^{[s]}u\big)_B \big|\,dx.
\end{equation}
Taking supremum over all balls $B$ in equation ~\eqref{E:ii-9}  results in:
\begin{equation}\label{E:ii-11}
    |u|_{\mathcal L ^{[s], \psi_{s,A}^{[s]}} (\rn)} \lesssim \big|\nabla^{[s]}u\big|_{\mathcal L^{[s],\psi_{s,A}^{ [s]}}(\rn)}.
\end{equation}
Since, according to the definitions \eqref{1} and \eqref{varphi_k}, 
\begin{equation*}
    \psi_{s,A}^{[s]}
    =
    \varphi_{\{s\},A},
\end{equation*}
equation ~\eqref{E:ii-11} reads:
\begin{equation}\label{E:ii-11bis}
    |u|_{\mathcal L ^{[s], \psi_{s,A}^{[s]}} (\rn)} \lesssim \big|\nabla^{[s]}u\big|_{\mathcal L^{\varphi_{\{s\},A}}(\rn)}.
\end{equation}
Consequently, an application of the inequality ~\eqref{3} with $u$ replaced with $\nabla^{[s]}u$ implies that
\begin{equation}\label{E:ii-12}\big|\nabla^{[s]}u\big|_{\mathcal L^{\varphi_{\{s\},A}}(\rn)} \lesssim \big|\nabla^{[s]}u\big|_{\{s\},A,\rn}, 
\end{equation}
and equation \eqref{204} now follows via 
~\eqref{E:ii-11bis} and~\eqref{E:ii-12}.
\\
To prove the optimality of the space $\mathcal L ^{[s], \psi_{{s},A}^{{[s]}}}(\rn)$ in~\eqref{203} and~\eqref{204}, assume that $\varphi$ is an admissible function such that
\begin{equation}\label{204bis}
   |u|_{\mathcal{L}^{[s], \varphi}(\rn)} \lesssim \, \big| \nabla ^{[s]} u\big|_{\{s\}, A, \rn}
\end{equation}
for $u\in  V^{s,A}(\rn)$. Then,
the  argument which 
yields \eqref{E:opt-6}
 in Part (i) can be repeated for $k=[s]$. It tells us that
\begin{equation}\label{E:opt-ii-1}
    \frac{1}{\|f\|_{L^A(0,\infty)}} \int_{|B|}^{\infty}f(r)\, r^{-1+\frac{\{s\}-1}{n}}\,dr
    \lesssim \varphi(|B|^{\frac{1}{n}})|B|^{-\frac{1}{n}}
\end{equation}
 for every  ball $B\subset \rn$ and every nonnegative function  $f \in L^A(0,\infty)$ which vanishes in $(0, |B|)$ and is non-increasing in $(|B|, \infty)$. 
Hence,
\begin{equation}\label{E:opt-ii-1a}
    \frac{|B|^{-1+\frac{\{s\}}{n}}}{\|f\|_{L^A(0,\infty)}} \int_{|B|}^{2|B|}f(r)\,dr
    \lesssim \varphi(|B|^{\frac{1}{n}})
\end{equation}
for every $B$ and $f$ as above.
Taking the supremum over all these functions $f$ yields:
\begin{equation}\label{E:opt-ii-1b}
    |B|^{-1+\frac{\{s\}}{n}}\|\chi_{(|B|,2|B|)}\|_{L^{\widetilde A}(0,\infty)} 
    \lesssim \varphi(|B|^{\frac{1}{n}})
\end{equation}
for every ball $B$.
Owing to ~\eqref{E:norm-characteristic} and~\eqref{varphi_k}, and to the arbitrariness of $B$, we deduce that
\begin{equation}\label{E:opt-ii-2tris}
    \psi_{s,A}^{[s]}(r)
    \lesssim \varphi(r)
    \qquad\text{for $r>0$.}
\end{equation}
The optimality of the space $\mathcal L^{[s],\psi_{s,A}^{[s]}}(\rn)$ in~\eqref{203} and~\eqref{204} follows from \eqref{E:opt-ii-2tris}. 

The equivalence of~\eqref{E:higher-vanishing-3}--\eqref{naples2} to ~\eqref{E:higher-vanishing-4} can be established  via an argument similar to that of the parallel statement in Theorem~\ref{teo2}, and is  omitted.
\end{proof}

\bigskip
\par\noindent
{\bf Acknowledgment}.  We wish to thank  the referees for very carefully checking the paper and for several comments and valuable suggestions.
 \smallskip

\section*{Compliance with Ethical Standards}\label{conflicts}

\smallskip
\par\noindent
{\bf Funding}. This research was partly funded by:
\\ (i) GNAMPA   of the Italian INdAM - National Institute of High Mathematics (grant number not available)  
(A. Alberico, A. Cianchi);
\\ (ii)\ Research Project   of the Italian Ministry of Education, University and
Research (MIUR) Prin 2017 \lq \lq Direct and inverse problems for partial differential equations: theoretical aspects and applications'',
grant number~201758MTR2 (A.~Cianchi);
\\ (iii)\ Research Project   of the Italian Ministry of Education, University and
Research (MIUR) Prin~2022 \lq \lq Partial differential equations and related geometric-functional inequalities'',
grant number 20229M52AS, cofunded by~PNRR (A.~Cianchi);
\\  (iv)  Grant no. 23-04720S of the Czech Science Foundation  (L. Pick and L. Slav\'{\i}kov\'{a});
\\ (v) Primus research programme PRIMUS/21/SCI/002 of Charles University (L. Slav\'ikov\'a);
\\ (vi) Charles University Research Centre program number~UNCE/24/SCI/005 (L.~Slav\'ikov\'a).

\bigskip
\par\noindent
{\bf Conflict of Interest}. The authors declare that they have no conflict of interest.

\end{document}